\documentclass[11pt,reqno]{amsart}
\usepackage{amsmath,amsthm,amsfonts, amssymb}
\usepackage{mathptmx}
\usepackage{mathtools} 
\usepackage{pgf, tikz}
\usepackage{tikz-cd}
\usepackage{forest}
\usepackage{graphicx}
\usepackage{hyperref}
\usepackage{cleveref} 

\usepackage[shortlabels]{enumitem}

\usepackage[left=1.2in, right=1.2in, top=1in, bottom=1in]{geometry}
\usepackage{xcolor}
\DeclareMathOperator{\Pic}{Pic}
\DeclareMathOperator{\Jac}{Jac}
\DeclareMathOperator{\Div}{Div}

\DeclareMathOperator{\snf}{SNF}
\DeclareMathOperator{\rank}{rk}

\newcommand\scalemath[2]{\scalebox{#1}{\mbox{\ensuremath{\displaystyle #2}}}}
\newcommand{\angles}[1]{\langle #1 \rangle}
  
\newcommand{\Z}{\mathbb{Z}}

\newcommand{\mycirc}[1][black]{\textcolor{#1}{\ensuremath\bullet}}

\theoremstyle{definition}
\newtheorem{mydef}{Definition}[section]
\newtheorem{myeg}[mydef]{Example}
\newtheorem{example}[mydef]{Example}

\newtheorem{question}[mydef]{Question}
\newtheorem{rmk}[mydef]{Remark}
\newtheorem{remark}[mydef]{Remark}

\theoremstyle{plain}
\newtheorem{mythm}[mydef]{Theorem}
\newtheorem{theorem}[mydef]{Theorem}
\newtheorem*{nothm}{Theorem}

\newtheorem{mytheorem}[mydef]{Theorem}
\newtheorem{lem}[mydef]{Lemma}
\newtheorem{lemma}[mydef]{Lemma}
\newtheorem{pro}[mydef]{Proposition}
\newtheorem{proposition}[mydef]{Proposition}

\newtheorem{corollary}[mydef]{Corollary}

\begin{document}

\title{On Picard groups and Jacobians of directed graphs}

\author{Jaiung Jun}
\address{Department of Mathematics, State University of New York at New Paltz, NY 12561, USA}
\email{junj@newpaltz.edu}

\author{Youngsu Kim}
\address{Department of Mathematics, California State University San Bernardino, San Bernardino, CA 92407}
\email{youngsu.kim@csusb.edu}

\author{Matthew Pisano}
\address{Computer Science Department, Rensselaer Polytechnic Institute, NY 12561, USA}
\email{pisanm2@rpi.edu}

\makeatletter
\@namedef{subjclassname@2020}{
\textup{2020} Mathematics Subject Classification}
\makeatother

\subjclass[2020]{05C50, 05C25, 05C20, 20K01}
\keywords{Jacobian of a graph, sandpile group, critical group, chip-firing game, directed graphs, cycle graph, wheel graphs, Laplacian of graph, trees, Smith normal form}

\begin{abstract}
The Picard group of an undirected graph is a finitely generated abelian group, and the Jacobian is the torsion subgroup of the Picard group. 
These groups can be computed by using the Smith normal form of the Laplacian matrix of the graph or by using chip-firing games associated with the graph. 
One may consider its generalization to directed graphs based on the Laplacian matrix. 
We compute Picard groups and Jacobians for several classes of directed trees, cycles, wheel, and multipartite graphs. 
\end{abstract}

\maketitle

\section{Introduction}

In this paper, we compute Picard groups and Jacobians for several classes of directed graphs, including undirected graphs as directed graphs having all bi-directional edges. 
Our results are based on the explicit description of Smith normal forms of Laplacian matrices, 
and our proofs utilize the properties of both graphs and algebra, including classical theorems such as Cramer's rule. 
Smith normal forms arise naturally in several area of mathematics.  
For instance, when computing homology groups of simplicial complexes and the structure theorem of a finitely generated abelian group\footnote{Or its generalization to the structure theorem of a finitely generated modules over a principal ideal domain}. 
Further, Stanley introduces Smith normal form and its applications in combinatorics in \cite{stanley2016smith}. 

For an undirected graph\footnote{All graphs are assumed to be finite. We allow multiple edges, but no loops.} $G$, 
one may define the Picard group and Jacobian of $G$ by using a combinatorial game, called a \emph{chip-firing game}, played on the graph $G$.
The game can be completely described by the Laplacian matrix $L_G$, and
this allows one to alternately define the Picard group and the Jacobian of $G$ from $L_G$ without referring the game.
In this paper, we use Laplacian matrices of directed graphs to define their Picard groups and Jacobians (Definitions \ref{definition: directed laplacian} and \ref{definition: directed picard groups and jacobians}). 

We first introduce the chip-firing game for undirected graphs. 
Consider an undirected graph $G$. 
To play a chip-firing game, one starts by placing chips (possibly negative as ``debt'') at each vertex of $G$. At each turn, a vertex borrows or lends chips from or to all its adjacent vertices simultaneously. 
One wins the game if one can reach at a chip configuration such that every vertex is debt-free after finitely many turns.

Whether one can win a game or not depends on the initial chip configuration. 
For instance, if the total number of chips is negative, 
then such a game is not winnable since borrowing and lending moves preserve the total number of chips. 
It is natural to ask whether one can determine a given chip configuration is winnable or not. 
In some cases, one can determine a winning chip configuration by using \emph{divisor theory on graphs} by a result of by Baker and Norine \cite[Theorem 1.9]{baker2007riemann}. 

The chip-firing game can be studied algebraically as follows.
One can write any chip configuration on a graph $G$ as an element of the free abelian group generated by the set of vertices $V(G)$ of $G$. 
The collection of all configurations is denoted by $\Div(G)$. 
An element of $\Div(G)$ is called a \emph{divisor}, 
and a divisor $D$ is \emph{effective} if it is of the form $D=\sum_{v \in V(G)} a_v v$, where $a_v \geq 0$. 
Divisors $D$ and $D'$ in $\Div(G)$ are \emph{equivalent} if $D'$ can be obtained from $D$ in a finite sequence of borrowing and lending moves. 
Under this correspondence, 
a chip configuration is {winnable} if and only if its corresponding divisor $D$ is equivalent to an effective divisor. 
This equivalence is in fact an congruence relation and defines a group 
\begin{equation}\label{picDefinition}
\Pic(G):=\Div(G)/\sim,
\end{equation}
called the \emph{Picard group} of $G$. 
The torsion subgroup of $\Pic(G)$, denoted by $\Jac(G)$, is called the \emph{Jacobian} of $G$.\footnote{Depending on the literature, $\Jac(G)$ is also called as a critical group or a sandpile group. See \cite[Section 5.6]{baker2007riemann}.} 

For an undirected graph $G$ with the vertices $V(G) = \{v_1,v_2,\dots,v_n\}$, one defines an $n\times n$ matrix $L_G$, called the \emph{Laplacian matrix} of $G$, as follows. 
\begin{equation*}
(L_G)_{ij}=\begin{cases}
\textrm{ degree of $v_i$} &\text{if}~ i=j;\\
-(\textrm{\# of edges between $v_i$ and $v_j$}) &\text{if}~ i \neq j,
\end{cases}
\end{equation*}
where $(L_G)_{ij}$ denotes the $(i,j)$th entry of the matrix $L_G$, 
and the degree of a vertex $v$ is the number of edges incident to $v$. 
For more details and background, refer to \cite[Section 2.1]{corry2018divisors}. 
We note that for directed graphs, one can define $L_G$ by using out-degrees (Definition \ref{definition: directed laplacian}).

One can use the Laplacian matrix $L_G$ to compute $\Pic(G)$ of a graph $G$ in \cref{picDefinition}:
Any chip configuration corresponds to a vector $\mathbf{v} \in \mathbb{Z}^{V(G)}$. 
The chip configuration $\mathbf{v}'$ obtained by a lending move (resp.\ borrowing move) at a vertex $v_i$ from $\mathbf{v}$ corresponds to the following computations in $\Z^{V(G)}$.  
\begin{equation}\label{eq: introduc1}
\mathbf{v}'=\mathbf{v} - L_G^T \, e_i \quad (\textrm{resp.\ } \mathbf{v}'=\mathbf{v} + L_G^T \, e_i),
\end{equation}
where $e_i$ denotes the $i$th standard (column) basis of $\Z^{V(G)}$ and the superscript $~^T$ denotes the transpose of a matrix.
Under this setting, the matrix $L_G^T$
corresponds to a linear map $\mathbb{Z}^{V(G)} \to \mathbb{Z}^{V(G)}$,
and one has the following isomorphisms. 
\begin{equation}\label{eq: two isomorphisms}
\Pic(G)\cong \mathbb{Z}\times \Jac(G)  \textrm{ and } \Pic(G)\cong \textrm{coker}(L_G^T).
\end{equation}
See \cite[Proposition 1.20, pp~17-18]{corry2018divisors} for the isomorphisms in \cref{eq: two isomorphisms}. 
Eq.~\eqref{eq: introduc1} allows one to extend the game to an arbitrary $n \times n$ integer matrix, replacing the role of $L_G$, 
and this direction of study was conducted under the name of \emph{Avalanche-finite} matrices in \cite[Section 6] {klivans2018mathematics}. 

The \emph{abelian sandpile model} is another combinatorial game on graphs. 
The rule is similar to the chip-firing game. 
A noticeable difference 
is that 
for the abelian sandpile model, each vertex should maintain a nonnegative number of chips throughout the game. 
The game stops when one can no longer fire chips. 
See \cite[Section 6]{corry2018divisors} for details and \cite{holroyd2008chip} for its generalizations for directed graphs.

In this paper, we define the \emph{Picard group} for a \emph{directed} graph $G$ by using the second isomorphism in \cref{eq: two isomorphisms}, and the \emph{Jacobian} of $G$ the torsion subgroup of the Picard group, see \Cref{definition: directed picard groups and jacobians}. 
Our definition generalizes the definition for undirected graphs, but we no longer have the isomorphism $\Pic (G) \cong \mathbb{Z} \times \Jac (G)$ in general. That is, the rank of $\Pic(G)$ need not be $1$. 

There has been a body of work devoted to the study of Picard groups and Jacobians for undirected graphs.
Though their extensions to directed graphs may diverge depending on one's purpose,
most generalizations agree for the class of a directed graphs having a \emph{global sink} (or a vertex that can play the role of a ``sink'')
\footnote{
Here the word ``sink'' means that a vertex of a directed graph that can be reached from any other vertex. 
In some books such as \cite[pp 201]{harary}, this is the defintion of a (global) sink.}, 
see \cite[Remark 2.12]{holroyd2008chip}.  
Several subclasses of our results satisfies this global sink condition.
Such a tight interplay between combinatorics and algebra was the one of the key motivations of our research. 

Our study focuses on the understanding both the free and torsion part of Picard groups. 
In particular, how these groups vary in families and how they compare to the Picard groups and Jacobians of their underlying directed graphs. 
Below we summarize our main results.

Wagner in \cite[Corollary 3.5]{wagner2000critical} proved that the rank of the Picard group of a directed graph is the number of terminal strong components. 
We prove that the Picard group of a directed tree is torsion free. 
By combining these results, we have the following theorem. 

\begin{nothm}[{\Cref{proposition: tree proposition}}]
The Picard group of any directed tree is free, and its rank is determined by the number of terminal strong component. 
\end{nothm}

The Picard group of the undirected cycle graph $C_n$ is $\Z \times \mathbb{Z}_n$ \cite[Corollary 4]{glass2020chip}. 
We prove that for any $k$ such that $1 \le k \le n$, there exists a directed cycle graph on $n$-vertices whose Picard group is isomorphic to $\Z \times \Z_k$. 

\begin{nothm}[{Theorem \ref{theorem: single term}}]
Let $n \geq 3$ and $C_n$ be a cycle graph with $n$ vertices. For any $1 \leq k \leq n$, there exists an orientation of $C_n$ such that $\Pic(C_n) \cong \Z \times \Jac(C_n)$, where $\Jac(C_n) = \mathbb{Z}_k$. 
\end{nothm}

In addition, we provide a way to determine the Picard group of a directed cycle graph having a global sink based on an invariant of the graph, see \Cref{uniqueDoubleChain}. 

Biggs determined the Jacobian of undirected wheel graphs in \cite[Theorem 9.2]{biggs1999chip}.
Consider the three classes of wheel graphs denoted by $W_n, W_n'$, and $W_n''$. 
The rims of all three classes have bi-directional edges.\footnote{In our notation the subscript $n$ denotes the number of vertices of a directed wheel graph.}
All the edges to the spokes in $W_n, W_n'$, and $W_n''$, are bi-directions, pointed to the axle, and pointed away from the axle, respectively.  
See \Cref{example: wheel graph orientations} for  pictorial examples. 
Biggs' result corresponds to the class $W_n$. 

\begin{nothm}[{Propositions \ref{proposition: wheel1} and \ref{proposition: wheel2}}]  
With the  above notations, we have the following. 
\begin{enumerate}[$(1)$]
\item 
The Laplacian matrices of $W_n$ and $W_n'$ are row equivalent. In particular, one has 
\begin{equation*}
\Pic (W_n) \cong \Pic (W_n').
\end{equation*}
\item 
One has
\[
\Pic(W_n'')\cong \begin{cases}
\mathbb{Z} \times  \mathbb{Z}_{n-1} \times \mathbb{Z}_{n-1}& \text{if}~ n ~\text{is even};\\
\mathbb{Z} \times \mathbb{Z}_{(n-1)/2} \times \mathbb{Z}_{2(n-1)} & \text{if}~ n ~\text{is odd}.
\end{cases}
\]
\end{enumerate}
\end{nothm}
It is interesting to note that the key arguments in item (2) use Cramer's rule. 
Indeed, we solve a certain quadratic equation whose solutions correspond to the minors of a matrix in question, and this explains the even and odd pattern, see \Cref{lemma: M_nCremer}.

Our last result is on certain multipartite directed graphs, named single-flow directed multipartite graphs. 
A \emph{single-flow directed multipartite graph} is a multipartite directed graph, 
where edges are directed to a single direction such as the forward propagation of a neural network. 
In the case where a single-flow directed multipartite graph has two or three layers, we provide a complete description of their Picard groups in terms of the number of vertices in each layer (\Cref{bipartitePic} and \Cref{multipartite3Pic}).


This paper is organized as follows. In Section \ref{section: preliminaries}, we review basic definitions and properties. 
In Sections \ref{section: trees}--\ref{section: conjecture}, we prove the results on tree, cycle, wheel and multipartite directed graphs listed above, respectively.

\medskip

\noindent\textbf{Acknowledgments:}~ This research was supported in part by the high performance computing resources provided by Information Technology Services at California State University San Bernardino.\footnote{This work was supported in part by NSF awards CNS-1730158, ACI-1540112, ACI-1541349, OAC-1826967, OAC-2112167, CNS-2120019, the University of California Office of the President, and the University of California San Diego’s California Institute for Telecommunications and Information Technology/Qualcomm Institute. Thanks to CENIC for the 100Gbps networks.} J.~Jun and M.~Pisano were partially supported by Research and Creative Activities (RSCA) at SUNY New Paltz. Lastly, several examples were computed with the help of {\tt{SageMath}} \cite{sagemath}. 

\section{Preliminaries}\label{section: preliminaries}

For a graph $G$, we use $V(G)$ and $E(G)$ to denote the vertex set and edge set of the graph $G$, respectively, and $V(G)$ is non-empty and finite. 
In the sequel, we treat undirected graphs as directed graphs having bi-directional edges. 
To emphasize this, we use the term \emph{arrow} for directed edges of directed graphs. 
Thus, an undirected graph is a directed graph with all its arrows are bi-directional. 
For a graph $G$, an arrow $e_{vw}$ between vertices $v$ and $w$ of $G$ can be one-directional or bi-directional. 
We use the notation of $e_{\overrightarrow{vw}}$ for the arrow from $v$ to $w$ and $e_{\overleftrightarrow{vw}}$ for a bi-directional arrow. 
For a vertex set $V(G) = \{ v_1,\dots, v_n \}$, we also use $\{ 1, \dots, n \}$ to denote the same vertex set.  

For a matrix $M$, we use the notation of $M = (m_{ij})$, where $m_{ij}$ denotes the $(i,j)$th entry of $M$.
The bold symbols $\mathbf{1}_k$ and $\mathbf{0}_k$ denote the row matrix of size $k$ consisting of $1$'s and $0$'s, respectively. 
We also use $\mathbf{1}$ and $\mathbf{0}$ if the size is clear from context. 
Lastly, the symbols $\mathbf{0}_{k \times k}$ and $\mathbf{1}_{k \times k}$ 
denote the zero matrix and the matrix consisting of $1$'s of size $k \times k$.

\begin{mydef}[{Laplacian matrix}]\label{definition: directed laplacian}
For a directed graph $G$ with vertex set $V(G) = \{v_1,v_2,\dots,v_n\}$, 
the \emph{Laplacian matrix} of $G$, denoted by $L_G = (l_{ij})$, is an $n\times n$ matrix whose $(i,j)$th entry is defined by 
\begin{equation*}
l_{ij}=\begin{cases}
\textrm{ the number of outgoing arrows of $v_i$} & \text{if}~ i=j,\\
-(\textrm{the number of arrows from $v_i$ to $v_j$}) & \text{if}~ i \neq j.
\end{cases}
\end{equation*}
\end{mydef}

\begin{remark}
For a directed graph $G$, the Laplacian $L_G$ equals $D_G - A_G$, 
where $D_G$ is the (out-)degree matrix and 
$A_G$ is the adjacency matrix of $G$. Note that for directed graphs, we only record the outgoing arrows in $D_G$ and $A_G$. 
\end{remark}

\begin{mydef}[{Sink} {\cite[pp 3]{holroyd2008chip}}]\label{defSink}
A vertex of a directd graph $G$ is called a \emph{sink} if it does not have any outgoing arrows (including bi-directional arrows). 
That is, the out-degree of the vertex is zero. 
We say a sink $v$ is \emph{global} (or a \emph{global sink}) 
if there exists a directed path from any vertex, which is not $v$, of $G$ to $v$. 
\end{mydef}

We note that if a directed graph has a global sink, then it is the unique sink of the directed graph.

Let $G$ be a directed graph. 
The \emph{underlying undirected graph} or simply \emph{underlying graph} is the graph $\overline{G}$ having the same vertex set as $G$ and for $v_i$ and $v_j$ of $\overline{G}$, there exist a bi-directional arrow between $v_i$ and $v_i$ if and only if there exists an arrow between $v_i$ and $v_i$ (regardless of the direction). 
By the \emph{orientation} of a graph, we mean the edge (or arrow) set of the graph.

\begin{myeg}\label{example: laplacian}
The graphs below have the same underlying undirected graph with different orientations. 
\begin{equation*}
\begin{array}{ccc}
T=\left(\begin{tikzcd}[row sep=0.4cm, column sep=0.4cm]
& 1 \arrow[d] & \\
2 \arrow[r]& 3 \arrow[u] \arrow[r]\arrow[l] \arrow[d] & 4 \arrow[l]\\
& 5 \arrow[u] &
\end{tikzcd} \right)
&T'=\left(\begin{tikzcd}[row sep=0.25cm, column sep=0.25cm]
& 1 \arrow[d] & \\
2 \arrow[r]& 3 & 4 \arrow[l]\\
& 5 \arrow[u] &
\end{tikzcd} \right)
&T'' = \left(\begin{tikzcd}[row sep=0.25cm, column sep=0.25cm]
& 1 & \\
2 & 3 \arrow[u] \arrow[r]\arrow[l] \arrow[d]& 4 \\
& 5  &
\end{tikzcd} \right)\\
~\\
L_T = \begin{bmatrix}
1& 0& -1&0 &0 \\
0 & 1 & -1 & 0 &0\\
-1& -1& 4& -1 &-1\\
0& 0&-1&1 & 0\\
0& 0&-1&0 & 1
\end{bmatrix}
&L_T' = \begin{bmatrix}
1& 0& -1&0 &0 \\
0 & 1 & -1 & 0 &0\\
0& 0& 0& 0 &0\\
0& 0&-1&1 & 0\\
0& 0&-1&0 & 1
\end{bmatrix}
& L_T'' = \begin{bmatrix}
0& 0& 0&0 &0 \\
0 & 0 & 0 & 0 &0\\
-1& -1& 4& -1 &-1\\
0& 0&0&0 & 0\\
0& 0&0&0 & 0
\end{bmatrix}.
\end{array}
\end{equation*}
\end{myeg}

\begin{mydef}[{Picard group and Jacobian}]\label{definition: directed picard groups and jacobians}
Let $G$ be a directed graph. 
The \emph{Picard group} $\Pic(G)$ is the cokernel of $L_G^T$ up to isomorphism. 
The \emph{Jacobian} $\Jac(G)$ of $G$ is the torsion subgroup of $\Pic(G)$. 
\end{mydef}

\begin{remark}
Since the Picard group is the cokernel of the map $L_G^T \colon \mathbb{Z}^{|V(G)|} \to \mathbb{Z}^{|V(G)|}$, it is a finitely generated abelian group. 
This in turn implies that Jacobian is also finitely generated.
\end{remark}

\begin{mydef}
We say $m$ by $n$ matrices $M$ and $N$ are \emph{equivalent} if there exist invertible matrices $P$ of size $m$ and $Q$ of size $n$ such that $M = PNQ$.
\end{mydef}

\begin{mydef}[Smith normal form]
Suppose $M \in \operatorname{Mat}_{n \times n}(R)$, where $R$ is a commutative ring.
The Smith normal form of $M$,
denoted by  $\snf(M) = (d_{ij})$, 
is an $n$ by $n$ diagonal matrices with entries in $R$ such that 
\begin{enumerate}
\item $\snf(M)$ is equivalent to $M$,
\item $d_{ii} | d_{i+1,i+1}$ for $i = 1, \dots, n-1$, and
\item $d_{ij} = 0$ if $i\neq j$. 
\end{enumerate}
\end{mydef}

\begin{remark}
The Smith normal form of a matrix is unique up to associates if exists. The existence of the Smith normal form of a matrix $M$ depends on $M$ and the ring $R$. 
Though $M$ need not have a Smith normal form in general, 
Smith normal forms exist for all matrices over a principal ideal domain, for instance, $\mathbb{Z}$.
Thus, one can compute the Picard group of a graph $G$, by computing the Smith normal form of $L_G^T$. 
\end{remark}

There exists an algorithm to compute the Smith normal form of a matrix over a principal ideal domain.
The following theorem is an ideal theoretic characterization of the Smith normal form, and we use this fact several times in this note.
For $M \in M_{m \times n}(R)$, 
$I_k(M)$ denotes the ideal generated by $k \times k$ minors of $M$, where $I_k(M) = 0$ if $k > \min\{m,n\}$ and $I_k = \angles{1}$ if $k \le 0$. 

\begin{mytheorem}[{cf. \cite[Theorem 2.4]{stanley2016smith}}]\label{theorem: gcd theorem}
Let $R$ be a principal ideal domain, $M \in \operatorname{Mat}_{n \times n}(R)$, and $\snf(M) = (d_{ij})$. 
Then we have the following. 
\begin{enumerate}[$(1)$]
\item 
For $1\leq k \leq n$, $I_k(M) = \langle d_{11} \cdots d_{kk} \rangle$ and 
\item $\snf(M) = \snf(M^T)$.
\end{enumerate}
\end{mytheorem}
\begin{proof}
Item (1) is in \cite[Theorem 2.4]{stanley2016smith}, and item (2) follows from item (1) since $I_k (M) = I_k (M^T)$ for any matrix $M$. 
\end{proof}

By this theorem, when determining Picard groups we use $L_G$ instead of $L_G^T$. 
The following remark is straightforward but useful for reducing the size of a  matrix. 

\begin{rmk}\label{remark: reduction}
Suppose $M$ and $N$ are matrices such that 
$N = \left[ \begin{array}{c|c}
M & * \\
\hline
0 & \pm 1
\end{array} \right]
~{or}~
\left[ \begin{array}{c|c}
M & 0 \\
\hline
* & \pm 1
\end{array} \right],$
where $*$ denotes arbitrary entries. 
For any $k$, $I_k(M) = I_{k+1}(N)$, and the cokernels of $M$ and $N$ are isomorphic. 
\end{rmk}

\begin{myeg}\label{example: example rank}
Let $T, T'$, and $T''$ be as in Example \ref{example: laplacian}. 
Their the Smith normal forms are the following matrices. 
\begin{equation*}
\snf(L_T) = \snf(L_{T'})=\left[\begin{array}{c|c}
I_4 & 0 \\ \hline
0 & 0
\end{array}\right]
~\text{and}~
\snf(L_{T''})=\left[\begin{array}{c|c}
I_1 & 0 \\ \hline
0 & \mathbf{0}_{4\times 4}
\end{array}\right].
\end{equation*}
In particular, $\Pic(T) \cong \Pic(T') \cong \mathbb{Z}$ and $\Pic(T'') \cong \mathbb{Z}^4$. 
\end{myeg}

\begin{theorem} \label{kernelLaplician}
Let $G$ be a directed graph. Then $L_G \cdot \mathbf{1}^T = \mathbf{0}$. 
\end{theorem}

\begin{remark}\label{laplaceFactUndirected}
In addition, if $G$ is connected and undirected, then we have 
\begin{enumerate}
\item 
$L_G$ is symmetric and $\mathbf{1}^T \cdot L_G = \mathbf{0}$ ({\cite[Theorem 1]{glass2020chip}}), 
\item 
$\rank L_G = |V(G)| -1$ (cf. \cref{eq: two isomorphisms}), and
\item 
$|\Jac(G)|$ is the number of spanning trees (\cite[Proposition 2.37 and Remark 2.38]{corry2018divisors}).
\end{enumerate}
\end{remark}

\section{Picard groups and Jacobians of directed trees}\label{section: trees}

In this section, we show that the Picard group of any directed tree is torsion free\footnote{A torsion free finitely generated abelian group is free.} (\Cref{proposition: tree proposition}), and it allows one to determine the Picard group (\Cref{picDirectedTree}), thanks to a result of Wagner (\Cref{theorem: wagner}). 


\begin{lem}\label{proposition: gluing an arrow proposition}
Let $G$ be a directed graph. Suppose that $G'$ is the directed graph having the following vertex set and edge set:
\[
V(G') = V(G) \sqcup \{ w \} \textrm{ and }E(G') = E(V) \sqcup \{ \alpha \},
\]
where $\alpha$ is either an arrow from $w$ to some $v$ in $E(V)$ or a bi-directional arrow. Then we have
\[
\Pic(G) \cong \Pic(G').
\]
\end{lem}
\begin{proof}
Let $|V(G)|=n$. 
We label the vertices of $G$ as $v_1,v_2,\dots,v_n$. 
First, suppose $\alpha = e_{\overrightarrow{wv}}$. 
Let $L_G=(l_{ij})$ (resp.\ $L_{G'}$) be the Laplacian matrix of $G$ (resp.\ $G'$). 
Then the matrix $L_{G'}$ is of the following form.
\begin{equation}\label{eq: arrow adding matrix}
L_{G'}=\left[\begin{array}{ccc|c|c}
l_{11}&l_{12}&\cdots &l_{1n}&0\\
l_{21}&l_{22}&\cdots &l_{2n}&0\\
\vdots & \vdots &\ddots & \vdots & \vdots \\ \hline
\vdots & \vdots & \cdots&l_{nn} & 0\\ \hline
0&0&\cdots &-1&1\\
\end{array}\right].
\end{equation}
Thus, $L_{G'}$ is of the form 
\begin{equation*} 
\left[\begin{array}{c|c}
L_G & 0 \\
\hline
* & 1
\end{array}\right],
\end{equation*}
and one has $\Pic(G) \cong \Pic(G')$ by \Cref{remark: reduction}.

Next, suppose $\alpha = e_{\overleftrightarrow{wv}}$.
Then we obtain the following Laplacian matrix for $G'$.
\begin{equation}\label{eq: eq two-sided}
L_{G'}=\left[\begin{array}{ccc|c|c}
l_{11}&l_{12}&\cdots &l_{1n}&0\\
l_{21}&l_{22}&\cdots &l_{2n}&0\\
\vdots & \vdots &\ddots & \vdots & \vdots \\ \hline
\vdots & \vdots & \cdots&l_{nn}+1 & -1\\ \hline
0&0&\cdots &-1&1\\
\end{array}\right].
\end{equation}
After adding the last row to the second last row, 
the matrix in \cref{eq: eq two-sided} becomes the matrix in \cref{eq: arrow adding matrix}. 
Therefore, one has $\Pic(G) \cong \Pic(G')$, and this completes the proof.
\end{proof}

For undirected trees $T$, one has $\Pic(T)\cong \mathbb{Z}$. 
This directly follows from the matrix-tree theorem (\cite[Theorem 9.3]{corry2018divisors}). 
However, for directed trees, the rank of $\Pic(T)$ can be arbitrarily large depending the number of terminal strong components (\Cref{definition: terminal strong}) of $T$. 
We show that Picard groups are torsion free in both cases.

\begin{theorem}\label{proposition: tree proposition}
Any directed tree graph $T$ has a torsion-free Picard group. 
That is $\Jac(T)=0$.
\end{theorem}

\begin{proof}
We use induction on the number of vertices. 
The base case consists of a single vertex with the empty edge set, 
and one has $\Pic (T) \cong \Z$ and $\Jac (T) = 0$. 

Suppose $T$ has $n$ vertices. Choose a leaf vertex\footnote{By a leaf vertex we mean a leaf vertex of the underlying undirected graph of $T$.} 
$v$ and let $T_{del}$ be the tree graph obtained by deleting the vertex $v$ and all the arrows incident to $v$. 
If the out-degree of $v$ is $1$, then by \Cref{proposition: gluing an arrow proposition}, we have
\[
\Pic (T) \cong \Pic (T_{del}) \textrm{ and } \Jac (T) \cong \Jac (T_{del}).
\]
By the induction hypothesis, we have $\Jac (T_{del}) = 0$, and hence $\Jac(T)=0$.

Now, assume that the out-degree of $v$ is $0$. 
Let $w$ be the vertex incident to $v$ and $e_{\overrightarrow{wv}}$ the arrow from $w$ to $v$. 
By labeling $w$ and $v$ to be the last two vertices, we have the following Laplacian matrix.
\begin{equation*}
L_{T}=\left[\begin{array}{ccc|c|c}
l_{11}&l_{12}&\cdots &l_{1,n-1}&0\\
l_{21}&l_{22}&\cdots &l_{2,n-1}&0\\
\vdots & \vdots &\ddots & \vdots & \vdots \\ \hline
\vdots & \vdots & \cdots&l_{n-1,n-1}+1 & -1\\ \hline
0&0&\cdots &0&0\\
\end{array}\right].
\end{equation*}

By subtracting a suitable multiple of the last column of $L_T$ from the rest, 
one sees that $L_T$ is equivalent to the left matrix below which in turn is equivalent to $L$ below:
\begin{equation*}
\left[\begin{array}{ccc|c|c}
l_{11}&l_{12}&\cdots &l_{1,n-1}&0\\
l_{21}&l_{22}&\cdots &l_{2,n-1}&0\\
\vdots & \vdots &\ddots & \vdots & \vdots \\ \hline
0 & 0 & \cdots & 1 & -1\\ \hline
0&0&\cdots &0&0\\
\end{array}\right]
\sim 
L = 
\left[\begin{array}{ccc|c|c}
l_{11}&l_{12}&\cdots &l_{1,n-1}&0\\
l_{21}&l_{22}&\cdots &l_{2,n-1}&0\\
\vdots & \vdots &\ddots & \vdots & \vdots \\ \hline
0 & 0 & \cdots & 1 & -1\\ \hline
0&0&\cdots &-1&1\\
\end{array}\right].
\end{equation*}
Notice that $L$ is the Laplacian matrix of a directed tree graph $T'$ which may or may not be connected.\footnote{We mean that the underlying graph $\overline{T'}$ may or may not be connected.}
Suppose first that $T'$ is not connected. 
The Picard group and Jacobian of $T'$ are the direct sums of the Picard groups and Jacobians of its connected component.
Since the connected components of $T'$ are also directed trees, by the induction hypothesis we are done. 
Suppose $T'$ is connected.
Let $T_{del}'$ denote the graph obtained by deleting the vertex $v$ from $T'$ and its adjacent arrow. 
Then $T'$ is obtained by attaching a bi-directional arrow to $T_{del}'$. 
By induction and \Cref{proposition: gluing an arrow proposition}
\begin{equation*}
\Pic(T') \cong \Pic (T_{del}')~\text{and}~
\Jac(T') \cong \Jac(T_{del}') = 0.
\end{equation*}
Now the result follows as $L$ is the Laplacian matrix of $T'$ and is equivalent to $L_T$.
This completes the proof.
\end{proof}

\begin{mydef}[Terminal strong component]\label{definition: terminal strong}
Let $G$ be a directed graph.
\begin{enumerate}
\item 
A \emph{strong component} $C$ of $G$ is a non-empty subgraph of $G$ such that for any pair of vertices $v_i$ and $v_j$ of $C$, there exist directed paths from $v_i$ to $v_j$ and from $v_j$ to $v_i$, respectively.
\item 
A strong component $C$ of $G$ is called \emph{terminal} if there is no arrow from any vertex $v$ of $C$ to another vertex in $V_G \setminus V_C$. 
\item 
$G$ is said to be \emph{strongly connected} if $G$ itself is a strong component. 
\end{enumerate}

\end{mydef}

\begin{example}
\begin{enumerate}
\item 
In each of the directed graphs below, the red subgraphs denote their strong terminal components. 
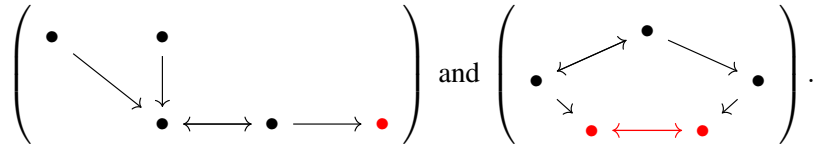
\begin{equation*}
\left(\begin{tikzcd}
\bullet  \arrow[dr,swap]
& \bullet \arrow[d,swap] \\
& \bullet\arrow[r,swap] & \bullet \arrow[l] \arrow[r,swap] & \mycirc[red]
\end{tikzcd} \right) 
~\text{and}~ 
\left( \begin{tikzcd}[row sep=0.2cm, column sep=0.2cm]
& & \bullet \arrow[drr] \arrow[dll]& &  \\ 
\bullet \arrow[urr] \arrow[dr] & & & & \bullet \arrow[dl] \\ 
& \mycirc[red] \arrow[rr,red] & & \mycirc[red] \arrow[ll,red]& 
\end{tikzcd}\right).
\end{equation*}

\item For graphs in \Cref{example: laplacian}, 
the terminal strong components are the subgraphs whose vertex sets are $V(T), \{3\}$, and $\{\{1\},\{2\},\{4\},\{5\}\}$, respectively.
\end{enumerate}
\end{example}

\begin{mytheorem}[{\cite[Corollary 3.5]{wagner2000critical}}] \label{theorem: wagner}
For any directed graph $G$, the rank of $\Pic(G)$ is the number of terminal strong components of $G$. 	
\end{mytheorem}

Combining the result of D.~Wagner and \Cref{proposition: tree proposition}, we have a characterization of the Picard group of a directed tree. 

\begin{corollary}\label{picDirectedTree}
Let $T$ be a directed tree 
and $r$ the number of terminal strong components of $T$.
Then the Picard group of $T$ is free of rank $r$. 
\end{corollary}

In the next two examples, we demonstrate that adding a vertex and an arrow may change the rank of the Picard group.

\begin{myeg}\label{example: oriented tree}
The directed tree $T$
\begin{equation*}
T=\left(\begin{tikzcd}[row sep=0.4cm, column sep=0.4cm]
1
& 2 \arrow[d,swap] &  3 \arrow[d] & \\
& 4\arrow[r,swap] \arrow[ul] & 5 \arrow[l] \arrow[u]\arrow[r,swap] & 6
\end{tikzcd}\right)
\end{equation*}
has $\Pic(T) \cong\mathbb{Z}^2$ and $\Jac(T)=0$.
\end{myeg}

\begin{myeg}
We add vertex $v_7$ and an arrow to the directed tree $T$ in \Cref{example: oriented tree} in two different ways. 
For $T'$, the rank of the Picard group increases by $1$. 
This can be seen from direction calculations or by observing that $T'$ has one more terminal strong component than $T$. 
For $T''$, the rank stays the same by \Cref{proposition: gluing an arrow proposition} or by counting the number of terminal strong compoenents. 
In all cases, their Picard groups are free. 
\begin{equation*}
T'=\left(\begin{tikzcd}[row sep=0.4cm, column sep=0.4cm]
1 	& 2 \arrow[d,swap] &  3 \arrow[d] & \\
7& 4\arrow[r,swap] \arrow[ul] \arrow[l,"\alpha"] & 5 \arrow[l] \arrow[u]\arrow[r,swap] & 6
\end{tikzcd}\right)
\qquad
T''=\left(\begin{tikzcd}[row sep=0.4cm, column sep=0.4cm]
1 	& 2 \arrow[d,swap] &  3 \arrow[d] & \\
7 \arrow[r,swap,"\beta"]& 4\arrow[r,swap] \arrow[ul]  & 5 \arrow[l] \arrow[u]\arrow[r,swap] & 6
\end{tikzcd}\right).
\end{equation*}
\end{myeg}

\section{Picard groups of cycle graphs}\label{section: cycles}

By a \emph{directed cycle graph}, we mean a directed graph whose underlying undirected group is a cycle graph. 
Let $C_n$ denote an arbitrary directed cycle graph with $n$-vertices. 
In this section, we prove a few theorems on Picard groups of directed cycle graphs $C_n$. 
When $C_n^{un}$ is undirected, $\Pic(C_n^{un}) \cong \Z \times \Z_n$ (\cite[Corollary 4]{glass2020chip}). 
We first show that for any $n$ and $k$ such that $1 \le k \le n$, there exists a directed cycle graph $C_n$ whose Jacobian is $\Z_k$ (\Cref{theorem: single term}).
Note that $\Z_1 = \{0 \}$. 
As an application, we identity a subclass for which there is a combinatorial invariant that determines their Jacobians while keeping the rank of the Picard group to be $1$ (\Cref{cycleGlobalSinkPic}). 
Throughout this section, we assume that $n\geq 3$ for $C_n$.

\begin{mytheorem}\label{theorem: single term}
For each $n$ and $k$ such that $1 \leq k \leq n$, there exists a directed cycle graph $C_n$ whose Picard group is isomorphic to $\Z \times \Jac(C_n)$, where $\Jac(C_n) \cong \Z_k$.
\end{mytheorem}

Indeed, we prove Theorem \ref{theorem: single term} by induction as follows:
Let $\sum_n$ be the set of all cycle graphs with $n$-vertices and $C_n^{un}$ the underlying undirected graph of $C_n$. We construct a map
\begin{equation}
\Phi \colon \sum_n \to \sum_{n+1}	
\end{equation} 
such that if $C_n \neq C_n^{un}$, then  
\begin{equation}\label{equation: pic iso}
\Pic(C_n) \cong \Pic(\Phi(C_n)).
\end{equation}
With this map and \Cref{example: example c3} (as a base case), to complete the proof by induction, one needs to find cycle graphs with $n+1$ vertices having $\Z_n$ and $\Z_{n+1}$, respectively. 
The cycle graph $C_{n+1}^{un}$ has $\Pic(C_{n+1}^{un})\cong \Z \times \mathbb{Z}_{n+1}$,
and we construct $C_{n+1}$ having $\Pic(C_{n+1}) \cong \Z \times \mathbb{Z}_n$ in \Cref{lemma: obj4}.
Based on the statements, in Subsection \ref{cyclecSubSection}, we show how to determine the Picard group and Jacobian of a directed cycle graph having a global sink. 

The following example serve as the base case.

\begin{myeg}\label{example: example c3}
With the following orientations of $C_3$
\begin{equation*}
G_1=\left(\begin{tikzcd}
\bullet \arrow[r] & \bullet \arrow[dl] \\
\bullet \arrow[u]& 
\end{tikzcd} \right) \quad G_2=\left(\begin{tikzcd}
\bullet  & \bullet \arrow[l]\arrow[dl] \\
\bullet \arrow[u]& 
\end{tikzcd} \right) \quad G_3=\left(\begin{tikzcd}
\bullet \arrow[r] \arrow[d]& \bullet \arrow[dl] \arrow[l]\\
\bullet \arrow[u] \arrow[ur]& 
\end{tikzcd} \right),
\end{equation*}
we have $\Jac(G_1)=0$, $\Jac(G_2)\cong \mathbb{Z}_2$, and $\Jac(G_3)\cong \mathbb{Z}_3$. 
In all three cases, the rank of the Picard group is $1$. 
This proves the case $n=3$ for \Cref{theorem: single term}. 
\end{myeg}

Recall that we use $e_{\overrightarrow{i j}}$ and $e_{\overleftrightarrow{i j}}$ to denote one-directional and bi-directional arrows from a vertex $v_i$ to a vertex $v_j$, respectively. 
We also use $e_{\overleftarrow{j i}}$ to denote $e_{\overrightarrow{i j}}$ if the order of vertices provides a better presentation.\footnote{We view a bi-directional arrow $e_{\overleftrightarrow{i j}}$ as a single arrow. Some authors view $e_{\overleftrightarrow{i j}}$ as a union of two one-directional arrows. 
In our work, our proofs use  induction on the number of vertices. 
This difference does not cause any issue and we believe that our choice was more suitable for our statements.}

Let $C_n$ be a directed cycle graph with $V(G) = \{ v_1, \dots, v_n \}$. 
If $C_n$ has a sink, then we may relabel the vertices of $C_n$ such that 
$v_n$ is the sink vertex with directed arrows $e_{\overrightarrow{n-1,n}}$ and $e_{\overleftarrow{n,1}}$.
Let $C_{n+1}'$ be the directed cycle graph whose vertex and edge sets are given as follows:
\begin{equation}\label{equation: C'}
V(C_{n+1}') = V(C_n) \sqcup \{ v_{n+1} \} ~
\text{and}~
E(C_{n+1}') =  E(C_n) 
\setminus \{e_{\overleftarrow{n,1}}\} 
\sqcup \{e_{\overrightarrow{n,n+1}}, e_{\overleftarrow{n+1,1}} \}.
\end{equation}
Pictorially, we have the following
\begin{equation*}
C_n= \left( \cdots \begin{tikzcd}[column sep=0.4cm]
v_{n-1} \arrow[r] & v_n & v_1 \arrow[l,red]
\end{tikzcd} \cdots \right)
\implies
C_{n+1}'= \left( \cdots \begin{tikzcd}[column sep=0.4cm]
v_{n-1} \arrow[r] & v_n \arrow[r,blue] & v_{n+1} &  v_1 \arrow[l,blue] 
\end{tikzcd} \cdots \right).
\end{equation*}
In this case, we call $C_{n+1}'$ a \emph{degree zero extension} of $C_n$.  

Suppose now that $C_n$ has a vertex of (out-)degree $1$. 
We may relabel the vertices of $C_n$ such that the vertex $v_n$ is of degree $1$ and $C_n$ has directed arrows $e_{\overrightarrow{n-1,n}}$ and $e_{\overrightarrow{n,1}}$ or that $C_n$ has directed arrows $e_{\overrightarrow{n-1,n}}$ and $e_{\overleftrightarrow{n,1}}$.
In both cases, we write $C_{n+1}''$ for the directed cycle graph whose edge set is given as follows.
\begin{equation}
V(C_{n+1}'') = V(C_n) \sqcup \{ v_{n+1} \} 
~\text{and}~
E(C_{n+1}'') = 
\begin{cases}
E(C_n) \setminus \{ e_{\overrightarrow{n,1}}\} \sqcup \{e_{\overrightarrow{n,n+1}}, e_{\overrightarrow{n+1,1}} \}
~\text{or}~\\
E(C_n) \setminus \{ e_{\overleftrightarrow{n,1}}\} \sqcup \{e_{\overrightarrow{n,n+1}}, e_{\overleftrightarrow{n+1,1}} \}.
\end{cases}
\end{equation}
Pictorially, we have the following
\begin{equation*}
\begin{aligned}
&C_n=\left( \cdots \begin{tikzcd}[column sep=0.4cm]
v_{n-1} \arrow[r] & v_n \arrow[r,red] & v_1 
\end{tikzcd} \cdots \right)
\implies
C_{n+1}''= \left( \cdots \begin{tikzcd}[column sep=0.4cm]
v_{n-1} \arrow[r] & v_n \arrow[r,blue] & v_{n+1} \arrow[r,blue] & v_1 
\end{tikzcd} \cdots \right)
~\text{or}~
\\
&C_n=\left( \cdots \begin{tikzcd}[column sep=0.4cm]
v_{n-1} \arrow[r] & v_n \arrow[r,red] & v_1 \arrow[l,red] 
\end{tikzcd} \cdots \right)
\implies
C_{n+1}''= \left( \cdots \begin{tikzcd}[column sep=0.4cm]
v_{n-1} \arrow[r] & v_n \arrow[r,blue] & v_{n+1} \arrow[r,blue] & v_1 \arrow[l,blue]
\end{tikzcd} \cdots \right).    
\end{aligned}
\end{equation*}
In this case, we call $C_{n+1}''$ a \emph{degree one extension} of $C_n$. 

Since a graph may have multiple vertices having degree $0$ or $1$ simultaneously, a single cycle graph $C_n$ may have both types of extensions, 
and these extensions depend on the choice of a vertex. 
However, we will show in \Cref{lemma: key lemma for cycles} that Picard groups are stable under these extensions, i.e., 
\[
\Pic (C_n) \cong \Pic (C_{n+1}') \textrm{ and } \Pic (C_n) \cong \Pic (C_{n+1}'').
\] 

\begin{myeg} 
Here are examples of degree $0$ and $1$ extensions of $C_5$ graphs. 
\begin{equation*}
\begin{aligned}
C_5=\left( \begin{tikzcd}[row sep=0.2cm, column sep=0.2cm]
& & 5  & &  \\ 
4 \arrow[urr] \arrow[dr] & & & & 1 \arrow[dl]\arrow[ull,red] \\ 
& 3 \arrow[rr] & & 2 \arrow[ll]& 
\end{tikzcd}\right)	
&\implies 
C_6'= \left( \begin{tikzcd}[row sep=0.2cm, column sep=0.3cm]
& 5 \arrow[rr,blue]&  &6 &  \\ 
4  \arrow[dr] \arrow[ur] & & & & 1 \arrow[dl] \arrow[ul,blue] \\ 
& 3 \arrow[rr] & & 2 \arrow[ll]& 
\end{tikzcd}\right)	\\
C_5=\left( \begin{tikzcd}[row sep=0.2cm, column sep=0.2cm]
& & 5 \arrow[drr,red] & &  \\ 
4 \arrow[urr] \arrow[dr] & & & & 1 \arrow[dl]\\ 
& 3 \arrow[rr] & & 2 \arrow[ll]& 
\end{tikzcd}\right)	
&\implies 
C_6''= \left( \begin{tikzcd}[row sep=0.2cm, column sep=0.3cm]
& 5 \arrow[rr,blue]&  &6\arrow[dr,blue] &  \\ 
4  \arrow[dr] \arrow[ur] & & & & 1 \arrow[dl]  \\ 
& 3 \arrow[rr] & & 2 \arrow[ll]& 
\end{tikzcd}\right)	
\end{aligned}
\end{equation*}
\end{myeg}
\vspace{0.3cm}

\begin{lem} \label{lemma: key lemma for cycles}
Let $C_n$ be a directed cycle graph. 
Suppose that not every arrow of $C_n$ is bi-directional. 
Then $C_n$ has a vertex $v$ of degree $0$ or $1$. 
Furthermore, the Picard group does not change under degree extensions at $v$, i.e., $\Pic (C_n) \cong \Pic (C_{n+1}')$ and $\Pic (C_n) \cong \Pic (C_{n+1}'')$. 
\end{lem}

\begin{proof}
Suppose that $C_n$ is a directed cycle graph such that not all arrows of $C_n$ are bi-directional. 
Let $V(C_n)=\{v_1,\dots,v_n\}$ and $D_{C_n} = (d_{ij})$ be the vertex set and the degree matrix of $C_n$, respectively. 
Since not every arrow is bi-directional, there exists $i$ such that $d_{ii} = 0$ or $1$. We may assume $i = n$ so that the adjacent vertices are $v_{n-1}$ and $v_1$.  

Suppose that $d_{nn} = 0$. In this case, the Laplacian matrix of $C_n$ is of the following form. 
\begin{equation*}
L_{C_n} = 
\begin{bmatrix}
l + 1 & -l & 0 & \cdots & 0 & -1 \\
\vdots & \vdots & \vdots & \cdots & \vdots &\vdots \\
0 & \cdots & 0 & -k & k+1 & -1 \\
0 & \cdots & 0 & 0 & 0 & 0 
\end{bmatrix}.
\end{equation*}
where $l, k \in \{ 0, 1 \}$.
Then we have the following Laplacian matrix $L_{C_{n+1}'}$ for $C_{n+1}'$. 
\begin{equation*}
L_{C_{n+1}'} = 
\begin{bmatrix} 
l + 1 & -l & 0 & \cdots & 0 & -1 \\
\vdots & \vdots & \vdots & \cdots & \vdots & \vdots  \\
0 & \cdots & -k & k+1 & -1 &0 \\
0 & \cdots & 0 & 0 & 1 & -1 \\
0 & \cdots & 0 & 0 & 0 & 0 
\end{bmatrix}.
\end{equation*}
Let $r_1,\dots,r_{n+1}$ and $c_1,\dots,c_{n+1}$ denotes the rows and columns of $L_{C_{n+1}'}$, respectively. 
After replacing $r_1$ with $r_1-r_n$, 
replacing $c_n$ with $c_n+c_{n+1}$, 
and switching $r_n$ and $r_{n+1}$ in sequence, 
we obtain the following matrix.
\begin{equation*}
\begin{bmatrix} 
l + 1 & -l & 0 & \cdots & -1 & 0 \\
\vdots & \vdots & \vdots & \cdots & \vdots & \vdots  \\
0 & \cdots & -k & k+1 & -1 &0 \\
0 & \cdots & 0 & 0 & 0 & 0 \\
0 & \cdots & 0 & 0 & 0 & -1 
\end{bmatrix}  = \left[ \begin{array}{c|c}
L_{C_n} & \mathbf{0} \\
\hline
\mathbf{0} & -1
\end{array} \right].
\end{equation*}
Now, the claim follows from Remark \ref{remark: reduction}.

Suppose $d_{nn} = 1$. First, we consider the case where the vertex $v_n$ has an outgoing arrow to $v_1$. 
Thus, we have the following Laplacian matrices for $C_n$ and $C_{n+1}''$, respectively.
\begin{equation}\label{degree1Type1}
L_{C_n} = \begin{bmatrix}
l  & -l & 0 & \cdots & 0 & 0 \\
\vdots & \vdots & \vdots & \cdots & \vdots & \vdots  \\
0 & \cdots & 0 & -k & k+1 & -1 \\
-1 & \cdots & 0 & 0 & 0 & 1 
\end{bmatrix},
\quad
L_{C_{n+1}''}=\begin{bmatrix} 
l  & -l & 0 & \cdots & 0 & 0 & 0 \\
\vdots & \vdots & \vdots & \cdots & \vdots & \vdots &\vdots \\
0 & \cdots & 0 & -k & k+1 & -1 & 0  \\
0 & \cdots & 0 & 0 & 0 & 1 & -1 \\ 
-1 & \cdots & 0 & 0 & 0 & 0 & 1 
\end{bmatrix}    
\end{equation}
where $l, k \in \{ 0, 1 \}$. 
Let $r_1,\dots,r_{n+1}$ and $c_1,\dots,c_{n+1}$ denote the rows and columns of $L_{C_{n+1}''}$, respectively.
By switching $r_{n+1}$ and $r_n$, 
replacing $r_n$ with $r_n+r_{n+1}$, 
and replacing $c_n$ with $c_n+c_{n+1}$ in sequence, 
we have the following matrix.
\begin{equation*}
\begin{bmatrix} 
l  & -l & 0 & \cdots & 0 & 0 & 0 \\
\vdots & \vdots & \vdots & \cdots & \vdots & \vdots &\vdots \\
0 & \cdots & 0 & -k & k+1 & -1 & 0  \\
-1 & \cdots & 0 & 0 & 0 & 1 & 0 \\ 
0 & \cdots & 0 & 0 & 0 & 0 & -1 \\ 
\end{bmatrix} = \left[ \begin{array}{c|c}
L_{C_n} & 0 \\
\hline
0 & -1
\end{array} \right].
\end{equation*}
Again, our claim follows from Remark \ref{remark: reduction}.

Finally, we assume that there exists a bi-directional arrow between the vertices $v_n$ and $v_1$. In this case by adding the last row of the Laplacian matrices  $L_{C_n}$ and $L_{C_n''}$ to the first one, we have the Laplacian matrices as in \cref{degree1Type1}.
This completes the proof. 
\end{proof}

\begin{lem} \label{lemma: obj4}
Let $C_n$ be a directed cycle graph with the following edge set. 
\begin{equation}\label{eq: orientation}
E(C_n)=
\{
e_{\overrightarrow{{n 1}}}, 
e_{\overleftarrow{{1 2}}}, 
e_{\overleftarrow{{2 3}}}\} 
\cup 
\{e_{\overleftrightarrow{i\, j}}
\mid i\neq j \textrm{ and } \{i,j\} \not \subseteq \{n,1,2,3\} \}.
\end{equation}
Pictorially, we have
\begin{equation*}
C_n= \left( \cdots \begin{tikzcd}[column sep=0.4cm]
v_n \arrow[r] & v_1 & v_2 \arrow[l] & v_3 \arrow[l]
\end{tikzcd} \cdots \right)
\end{equation*}
and all other arrows are bi-directional. 
Then we have
\[
\Pic (C_n) \cong \Z \times \Z_{n-1} \textrm{ and } \Jac (C_n) \cong \Z_{n-1}.
\]
\end{lem}
\begin{proof}
The Laplacian matrix of $C_n$ is of the following form.
\begin{equation*}
L_{C_n} = \begin{bmatrix}
0 & 0 & 0 & 0 & 0 & 0 & 0 & 0   \\
-1 & 1 & 0 & \cdots & 0 & 0 & 0 & 0  \\
0 & -1 & 2 & -1 & \cdots & 0 & 0 & 0 \\
0 & 0 & -1 & 2 & -1 & \cdots & 0 & 0  \\
\vdots & \vdots & \vdots & \vdots & \vdots & \ddots & \vdots& \vdots  \\
-1 & 0 & 0 & 0 & 0 & \cdots & -1 & 2 \\
\end{bmatrix}.
\end{equation*}
Let $r_1,\dots,r_{n}$ and $c_1,\dots,c_{n}$ denote the rows and columns of $L_{C_n}$, respectively.
Replace $r_n$ with $r_n-r_2$ and then $c_2$ with $c_2+c_1$. 
Further, switch $r_1$ and $r_2$. 
As a result, we obtain the following matrix.
\begin{equation*}
M_n=\begin{bmatrix}
-1 & 0 & 0 & \cdots & 0 & 0 & 0 & 0 \\
0 & 0 & 0 & 0 & 0 & 0 & 0 & 0   \\
0 & -1 & 2 & -1 & \cdots & 0 & 0 & 0 \\
0 & 0 & -1 & 2 & -1 & \cdots & 0 & 0  \\
\vdots & \vdots & \vdots & \vdots & \vdots & \ddots & \vdots& \vdots  \\
0 & -1 & 0 & 0 & 0 & \cdots & -1 & 2 \\
\end{bmatrix}. 
\end{equation*}
Now, it follows from Remark \ref{remark: obj2} below, one can observe that $M_n$ is equivalent to the following matrix:
\begin{equation*}
M_n'=\begin{bmatrix}
-1 & 0 & 0 & \cdots & 0 & 0 & 0 & 0 \\
0 & 2 & -1 & 0 & 0 & 0 & 0 & -1   \\
0 & -1 & 2 & -1 & \cdots & 0 & 0 & 0 \\
0 & 0 & -1 & 2 & -1 & \cdots & 0 & 0  \\
\vdots & \vdots & \vdots & \vdots & \vdots & \ddots & \vdots& \vdots  \\
0 & -1 & 0 & 0 & 0 & \cdots & -1 & 2 \\
\end{bmatrix} = \left[ \begin{array}{c|c}
-1 & 0 \\
\hline
0 & L
\end{array} \right], 
\end{equation*}
where the $(n-1) \times (n-1)$ bottom right submatrix denoted by $L$ is the Laplacian matrix of $C_{n-1}^{un}$, a directed cycle graph corresponding the undirected cycle graph on $(n-1)$-vertices. 
In particular, $\Pic(C_{n-1}^{un})\cong \Z_{n-1} \times \mathbb{Z}$. 
Now, the statement follows from \Cref{remark: reduction}.
\end{proof}

\begin{rmk} \label{remark: obj2}
The Laplacian of the undirected cycle graph $C_n$ is of the form 
\begin{equation*}
L = 
\begin{bmatrix}
2 & -1 & 0 & 0 & 0 & 0 & -1   \\
-1 & 2 & -1 & \cdots & 0 & 0 & 0 \\
0 & -1 & 2 & -1 & \cdots & 0 & 0  \\
\vdots & \vdots & \vdots & \vdots & \vdots & \vdots & \vdots  \\
-1 & 0 & 0 & 0 & \cdots & -1 & 2 \\
\end{bmatrix}.
\end{equation*}
Since $[1 \cdots 1] L = 0$  by \Cref{laplaceFactUndirected}, 
the first row is a $\Z$-linear combination of the next $(n-1)$ rows. 
This justifies the last equivalence in the proof above.  
\end{rmk}

\begin{rmk}
We note that when $n=3$, $C_n$ in \Cref{lemma: obj4} corresponds to $G_2$ in \Cref{example: example c3}.
\end{rmk}

\begin{proof}[Proof of Theorem \ref{theorem: single term}]
We induct on the number of vertices. 
The base case is when $n = 3$ which is Example \ref{example: example c3}. 
Suppose that the statement holds for $n$. 
For $n+1$ and for $1 \leq k \leq  n-1$, 
by the induction hypothesis on $C_n$, there exists a directed cycle graph $C_n$ having $\Pic(C_n) \cong \Z \times \mathbb{Z}_k$. 
Furthermore, not all arrows of $C_n$ are bi-directional; if so, then $\Jac(C_n)\cong\mathbb{Z}_n$. 
Therefore, by \Cref{lemma: key lemma for cycles}, there exists a directed cycle $C_{n+1}$ which is a degree extension of $C_n$ so that $\Pic (C_{n+1}) \cong \Pic(C_n) \cong \Z \times \mathbb{Z}_k$. 

By Lemma \ref{lemma: obj4}, there exists a directed cycle graph $C_{n+1}$ whose Picard group is  $\Z \times \mathbb{Z}_{n}$, 
and $C_{n+1}^{un}$, corresponding to the undirected cycle graph, has $\Pic (C_{n+1}^{un}) \cong \Z \times \mathbb{Z}_{n+1}$. 
This completes the proof. 
\end{proof}

For any given $n \ge 3$, let $\Gamma (n)$ denote the set of all Picard groups of $C_n$. 
Indeed, the proof shows that $\Gamma (n) \subset \Gamma (n+1)$. 
It would be interesting to study the numbers associated to $\Gamma (n), \Gamma (n+1)$, and $\Gamma (n+1) \setminus \Gamma (n)$ as $n$ grows infinity. 

\begin{question}
What are the members of $\Gamma(n)$? Is there a pattern that depends on $n$ only? 
What about the case for their Jacobians?
\end{question}

\subsection{Cycle graphs having a global sink}\label{cyclecSubSection}

In this subsection, we present an application of our results on directed cycle graphs having a global sink. 
We show that these graphs have a rank one Picard group 
and their Jacobian is determined by the number of bi-directional arrows in a certain position (\Cref{cycleGlobalSinkPic}). 
We first prove the statement for a special case (\Cref{theorem: two path}). 

\begin{mydef}\label{definition: path}
Let $C_n$ be a directed cycle graph. 
\begin{enumerate}
\item 
By a \emph{path} of $C_n$, we mean a connected subgraph of $C_n$ in which all arrows are oriented in a single direction. 
We do not allow any bi-directional arrows in a path. 

\item 
By a \emph{$C_n$ with two opposite paths}, we mean a directed cycle graphs $C_n$ which has exactly two paths sharing a sink, i.e., 
$C_n$ has two paths 
$P_1=(v_i \to \cdots \to v_j)$ and 
$P_2=(v_\ell \gets \cdots \gets v_k)$ such that $v_j=v_\ell$, and all the other arrows not in $P_1$ and $P_2$ are bi-directional.
\end{enumerate}
\end{mydef}

\begin{myeg}\label{twoPathExample}
The following is an example of $C_5$ with two opposite paths.
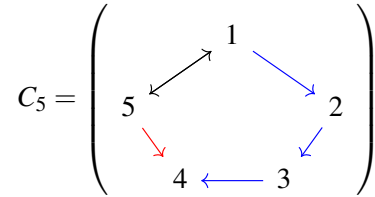
\begin{equation*}
C_5=\left( \begin{tikzcd}[row sep=0.4cm, column sep=0.15cm]
& & 1 \arrow[drr,blue] \arrow[dll]& &  \\
5 \arrow[urr] \arrow[dr,red] & & & & 2 \arrow[dl,blue] \\
& 4& & 3 \arrow[ll,blue]& \end{tikzcd} \right)
\end{equation*}	
We have the following two paths.
\begin{equation*}
P_1=\left( \begin{tikzcd}[row sep=0.3cm, column sep=0.3cm]
5 \arrow[rr] &  &4
\end{tikzcd}\right), \qquad 
P_2=\left( \begin{tikzcd}[row sep=0.4cm, column sep=0.4cm]
4  & 3 \arrow[l] & 2 \arrow[l] & 1 \arrow[l]
\end{tikzcd}\right).
\end{equation*}
\end{myeg}

\begin{remark}
Any $C_n$ with two opposite paths has a global sink. 
Therefore, for these graphs the notions of different chip firing games agree, see \cite[Remark 2.12]{holroyd2008chip}.
\end{remark}

\begin{mythm}\label{theorem: two path}
Let $C_n$ be a directed cycle graph with two opposite paths. Then we have
\begin{equation*}
\Pic (C_n) \cong \Z \times \Z_{k+2} ~\text{and}~ \Jac (C_n) \cong \Z_{k+2}.
\end{equation*}
Here, $k$ denotes the number of bi-directional arrows in $C_n$. 
\end{mythm}

\begin{myeg}\label{example: counting paths}
Consider the following orientations of $C_5$ with two paths. The counter-clockwise path is in red and the clockwise path is in blue.
The bi-directional arrows are shown in black.
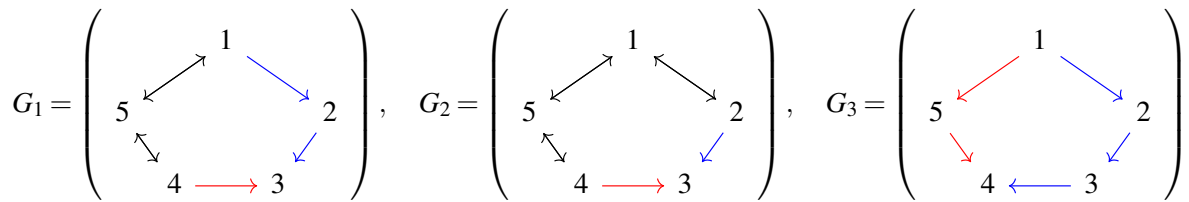
\begin{equation*}
\quad G_1=\left( \begin{tikzcd}[row sep=0.4cm, column sep=0.15cm]
& & 1 \arrow[drr,blue] \arrow[dll]& &  \\
5 \arrow[urr] \arrow[dr] & & & & 2 \arrow[dl,blue] \\
& 4\arrow[rr,red]\arrow[ul] & & 3 &
\end{tikzcd} \right),
\quad G_2=\left( \begin{tikzcd}[row sep=0.4cm, column sep=0.15cm]
& & 1 \arrow[drr] \arrow[dll]& &  \\
5 \arrow[urr] \arrow[dr] & & & & 2 \arrow[dl,blue] \arrow[ull] \\
& 4\arrow[rr,red]\arrow[ul] & & 3 &
\end{tikzcd} \right),
\quad G_3=\left( \begin{tikzcd}[row sep=0.4cm, column sep=0.15cm]
& & 1 \arrow[drr,blue] \arrow[dll,red]& &  \\
5  \arrow[dr,red] & & & & 2 \arrow[dl,blue] \\
& 4 & & 3 \arrow[ll,blue]&
\end{tikzcd} \right)
\end{equation*}
Then we have $\Jac(C_5)\cong\mathbb{Z}_3$, $\Jac(G_1)\cong\mathbb{Z}_4$, $\Jac(G_2)\cong\mathbb{Z}_5$, and $\Jac(G_3)=\mathbb{Z}_2$, where $C_5$ is the graph in \Cref{twoPathExample}.
All these graphs have rank $1$ Picard groups.  
\end{myeg}

\Cref{lemma: obj3} will be useful for the proof of \Cref{theorem: two path}.

\begin{lemma} \label{lemma: obj3}
For $n \ge 2$, let $M_n$ denote the $n \times n$ matrix whose diagonal entries are $2$ and sub-diagonals are $-1$, i.e., 
\begin{equation}\label{eq: m_n}
M_n = \begin{bmatrix}
2 & -1 & 0 & 0& \cdots & \cdots & 0 \\
-1 & 2 & -1& 0& \cdots &\cdots & 0 \\
0 & -1 & 2 & -1 & 0& \cdots & 0 \\
\vdots & \vdots & \vdots & \vdots & \vdots& \vdots & \vdots \\
0 & 0 & \cdots& 0 & -1 & 2 & -1 \\
0 & 0 &0 & \cdots & 0 &  -1 & 2 
\end{bmatrix}.
\end{equation}
Then one has $\det (M_n) = n+1$. 
\end{lemma}

\begin{proof}
We use induction on the size of $M_n$. When $n=2$, we have
\begin{equation*}
M_2 = \begin{bmatrix}
2 & -1 \\
-1 & 2
\end{bmatrix}
\end{equation*}
whose determinant is $3$. 

Suppose the statement is true for all $k \le n-1$. 
To compute $\det (M_n)$, we use the Laplace expansion along the first row.
Thus, one has 
\begin{equation*}
\det (M_n) = 2 \det (M_{n-1}) + \det(N),
\end{equation*}
where $N$ is the following $(n-1)\times(n-1)$ matrix.
\begin{equation}
N=\left[ \begin{array}{c|c c c c c  }
-1 &  -1& 0 & \cdots & 0 \\
\hline 
\mathbf{0} &  &  M_{n-2}
\end{array} \right].
\end{equation}
In particular, $\det(N)=-\det(M_{n-2})$ (cf. by \Cref{remark: reduction}), and hence by the induction hypothesis, we have 
\begin{equation*}
\det(M_n)=2\det(M_{n-1})+\det(N) = 2((n-1)+1) + (-1)((n-2)+1)=n+1. \qedhere
\end{equation*}
\end{proof}

\begin{rmk}\label{remark: key}
Note that the $(n-1) \times (n-1)$ minor of $M_n$ 
after deleting the first row and last column is $(-1)^{n-1}$. 
Hence the Smith normal form of $M_n$ is 
\begin{equation*}
\left[ \begin{array}{c|c}
I_{n-1} & 0 \\
\hline 
0 & n+1 
\end{array} \right].
\end{equation*}
\end{rmk}

\begin{lemma} \label{lemma: final lemma}
Let $C_n$ be a directed cycle graph and $V(C_n)=\{v_1,\dots,v_n\}$. Suppose that the vertex $v_1$ does not have any outgoing arrows, and all other vertices have two outgoing arrows. 
Then we have $\Pic(C_n) \cong \Z \times \Jac(C_n)$, where $\Jac (C_n) \cong \mathbb{Z}_n$.
\end{lemma}

\begin{proof}
The Laplacian matrix of $C_n$ is the matrix following.
\begin{equation}
L_{C_n}=\begin{bmatrix}
0 & 0 & \cdots  &\cdots  &0 &0 \\
-1 & 2 & -1 & \cdots & \cdots & 0 \\
0 & -1 & 2 & -1 & \cdots & 0\\
\vdots &  \vdots & \vdots & \vdots & \vdots &\vdots \\
0 &  \cdots& 0 & -1 & 2 & -1 \\
-1 &  0 & \cdots & 0 &  -1 & 2 
\end{bmatrix}. 
\end{equation}
By Remark \ref{kernelLaplician}, $L_{C_n}$ is equivalent to the following matrix
\begin{equation}
\left[ \begin{array}{c|c}
0 & 0 \\
\hline 
0 & M_{n-1}
\end{array} \right],
\end{equation}
where $M_{n-1}$ is the matrix in Lemma \ref{lemma: obj3}.
Thus, by Remark \ref{remark: key}, we have
\begin{equation*}
\Pic(C_n) \cong \Z \times \mathbb{Z}_n ~\text{and}~ \Jac (C_n) \cong \mathbb{Z}_n. \qedhere
\end{equation*}
\end{proof}

\begin{proof}[Proof of Theorem \ref{theorem: two path}]
The graph $C_n$ has a global sink. 
Recall $k$ denotes the number of bi-directional arrows. 
We induct on $n$. 
When $n=3$, $G_2$ in \Cref{example: example c3} verifies the case of $k = 0$. 
For $C_3$ with two opposite paths having one bi-directional arrow, its Laplacian matrix and its Smith normal form are the following.
\begin{equation*}
\begin{bmatrix}
0 & 0 & 0 \\
-1 & 2 & -1 \\
-1 & -1 & 2 
\end{bmatrix}
~\text{and}~
\begin{bmatrix}
1 & 0 & 0 \\
0 & 3 & 0 \\
0 & 0 & 0 
\end{bmatrix}.
\end{equation*}
This proves the base case. 

Suppose the statement holds true for two opposite paths directed cycle graphs on $n-1$ vertices. For the directed cycle graphs $C_n$ on $n$ vertices, if $k = n-2$, then we are done by \Cref{lemma: final lemma}.
If $k < n-2$, our $C_n$ can be obtained as a degree $0$ extension of a $C_{n-1}$ with two opposite paths. 
Both $C_n$ and $C_{n-1}$ are with two opposite paths having the same number of bi-directional arrows. 
Thus, by \Cref{lemma: key lemma for cycles} and induction,
we have 
$\Pic(C_{n}) \cong \Pic(C_{n-1}) \cong \Z \times \Z_{k+2}$.  
This completes the proof. 
\end{proof}

\Cref{strighteningLemma} is easy to prove but key to the proof of \Cref{cycleGlobalSinkPic}. 
It is in the same vein as \Cref{proposition: gluing an arrow proposition}. 
\begin{lemma}\label{strighteningLemma}
Let $G$ be a directed graph having and $v,w$ vertices of $G$. 
Suppose $e_{\overleftrightarrow{vw}}$ is in $E(G)$, and the (out-)degree of $v$ is $1$. 
Let $G'$ be a directed graph such that 
\begin{equation*}
V(G') = V(G) ~\text{and}~ E(G') = E(G) \setminus \{ e_{\overleftrightarrow{vw}} \} \cup \{ e_{\overrightarrow{vw}} \}.
\end{equation*}
\end{lemma}
Then $L_G$ and $L_{G'}$ are row equivalent. Thus, they have the same Picard group and Jacobian.
\begin{proof}
Without loss of generality, assume $v, w$ be the first two vertices. 
Then their Laplacians only differ in the second row. 
The row equivalence now follows by subtracting the first row of $L_G$ to the second. 
\end{proof}

\begin{myeg} 
Consider the following directed graph. 
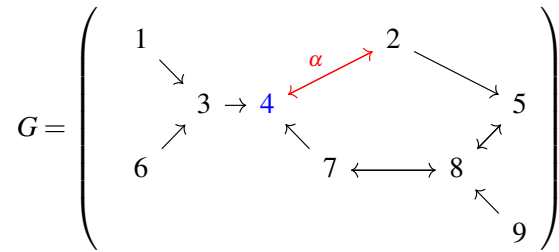
\[
G=\left(\begin{tikzcd}[row sep=0.3cm, column sep=0.3cm]
&1 \arrow[dr] & &	& & 2 \arrow[drr] \arrow[dll,swap,"\alpha",red] & &  \\ 
&& 3 \arrow[r] &	\textcolor{blue}{4} \arrow[urr,red] & & & & 5 \arrow[dl]\\ 
& 6 \arrow[ru] & &	& 7 \arrow[rr]  \arrow[ul]& & 8 \arrow[ur] \arrow[ll] & \\
& && & & && 9 \arrow[ul]
\end{tikzcd} \right)
\]	
Then, we have $\Pic(G) \cong \mathbb{Z} \times \mathbb{Z}_2$. Change $\alpha$ to an one-directional arrow to obtain the following graph:
\[
G'=\left(\begin{tikzcd}[row sep=0.3cm, column sep=0.3cm]
&1 \arrow[dr] & &	& & 2 \arrow[drr]  & &  \\ 
&& 3 \arrow[r] &	\textcolor{blue}{4} \arrow[urr,red] & & & & 5 \arrow[dl]\\ 
& 6 \arrow[ru] & &	& 7 \arrow[rr] \arrow[ul]& & 8 \arrow[ur] \arrow[ll] & \\
& && & & && 9 \arrow[ul]
\end{tikzcd} \right)
\]	
Then, we have $\Pic(G') \cong \mathbb{Z} \times \mathbb{Z}_2$.
\end{myeg}


\begin{mydef}
By a \emph{double chain graph}, we mean a directed graph $G$ whose underlying graph is a linear (or path) graph such that for the vertex set $\{ v_0, \dots, v_{k+1} \}$ of $G$, 
the edge set is defined as follows.
\begin{equation*}
E(G) = \{ e_{\overleftarrow{0,1}},
e_{\overleftrightarrow{1,2}}, \dots, 
e_{\overleftrightarrow{k-1,k}},
e_{\overrightarrow{k,k+1}} \}.
\end{equation*}
Pictorially, $G$ is the graph
\begin{equation*}
G=\left( \begin{tikzcd}[column sep=0.4cm]
v_{0}  
& v_1 \arrow[r] \arrow[l,red]
& v_2 \arrow[r] \arrow[l] 
& \cdots \arrow[r] \arrow[l] 
& v_{k-1} \arrow[r] \arrow[l] 
& v_k \arrow[r,blue] \arrow[l] 
& v_{k+1} 
\end{tikzcd} \right).
\end{equation*}
In this case, we say $G$ is a double chain of \emph{length} $k$.
\end{mydef}

\begin{lemma}\label{uniqueDoubleChain}
Let $C_n$ be a directed cycle graph. If $C_n$ has a global sink, then there exists at most one double chain subgraph contained in $C_n$. 
\end{lemma}
\begin{proof}
Suppose $C_n$ has more than one double chain subgraphs. 
Choose distinct double chain subgraphs and label them as $G$ and $G'$. 
Let $V(G) = \{ v_0, \dots, v_{k+1} \}$ and $V(G') = \{ w_0, \dots, w_{l+1} \}$. 
Further, let $v$ denote the global sink vertex. 
Without loss of generality, we may assume that $v_{k+1} \le v \le w_0$ in the vertex order of $C_n$. 
Then there does not exist a directed path from $w_{l+1}$ to the global sink vertex $v$. 
If there is a directed path $P$ from $w_{l+1}$ to $v$, then $w_l \not\in P$. 
This implies that $G$ is a subgraph of $P$. 
In particular, $v_0$ and $v_1$ are in $P$. 
Since the there is no arrow from $v_0$ to $v_1$, 
$v_1$ cannot be in $P$. 
This is a contradiction. 
\end{proof}

By \Cref{uniqueDoubleChain}, a directed cycle graph $C_n$ having a global sink, we can say \emph{the} length of the double chain of $C_n$. 
This length determines the Picard group and Jacobian of such $C_n$. 

\begin{theorem}\label{cycleGlobalSinkPic}
Let $C_n$ be a directed cycle graph with a global sink.
Suppose $k$ is the length of the \emph{double chain} subgraph contained in $C_n$. 
Then one has 
\begin{equation*}
\Pic(C_n) \cong \Z \times \Jac(C_n) \textrm{ and } \Jac(C_n)\cong \mathbb{Z}_{k+2}.
\end{equation*}
\end{theorem}
\begin{proof}
Let $C_n$ be a directed cycle graph with a global sink. 
We claim that there exists another directed cycle graph $C_n'$ such that $C_n'$ is with two opposite paths having the same double chain subgraph and the Laplacian of $C_n'$ is (row)-equivalent to the Laplacian of $C_n$. 
Once we have shown the claim, the proof will follow from \Cref{theorem: two path}.
By \Cref{strighteningLemma}, we know that the following replacement of an arrow keeps the Laplacian matrix up to row equivalence. 
\begin{equation*}
\left( \begin{tikzcd}[column sep=0.4cm] 
v_{i-1}  \arrow[r]
& v_i \arrow[r] 
& v_{i+1}   \arrow[l] 
\end{tikzcd} \right)
\rightarrow
\left( \begin{tikzcd}[column sep=0.4cm]
v_{i-1}  \arrow[r]
& v_i \arrow[r] 
& v_{i+1}   
\end{tikzcd} \right).
\end{equation*}
By inductively applying this procedure to all bi-directional arrows that are not part of the double chain, we will arrive to the desired $C_n'$. 
This completes the proof.
\end{proof}

\section{Picard groups of directed wheel graphs} \label{section: wheel graphs}
In this section, we use the following notation for directed wheel graphs and determine their Picard groups.

\begin{mydef}[{Wheel graph}]
By the directed wheel graph $W_n$, we mean a directed graph obtained by connecting a single universal vertex to all vertices of a cycle graph $C_{n-1}$ whose arrows are bi-directional arrows. 
This is the graph corresponding to the undirected wheel graph. 
\begin{enumerate}
\item 
By $W_n'$, we mean a directed wheel graph such that the arrows of the rim are bi-directional and all its spoke arrows point to the axle.
\item 
By $W_n''$, we mean a directed wheel graph such that the arrows of the rim are bi-directional and all its spoke arrows point away from the axle.
\end{enumerate}
\end{mydef}

Note that $W_n'$ and $W_n''$ have the same underlying undirected graph $W_n$. 

\begin{myeg}\label{example: wheel graph orientations}
Here are examples demonstrating our definitions. 
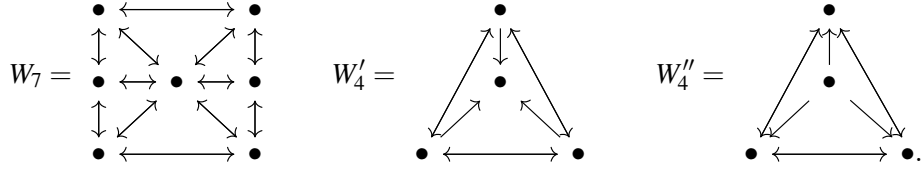
\begin{equation*}
W_7=\begin{tikzcd}[row sep=0.5cm, column sep=0.5cm]
\bullet \arrow[dr] \arrow[d]\arrow[rr]& &\bullet \arrow[d] \arrow[ll] \arrow[dl]\\
\bullet \arrow[r]\arrow[u]\arrow[d]& \bullet \arrow[l] \arrow[ru] \arrow[lu] \arrow[dl] \arrow[dr] \arrow[r]& \bullet\arrow[d] \arrow[l] \arrow[u]\\
\bullet \arrow[rr] \arrow[u] \arrow[ru] & &\bullet \arrow[u] \arrow[ul] \arrow[ll]
\end{tikzcd} \qquad W_4'=\begin{tikzcd}[row sep=0.5cm, column sep=0.5cm]
& \bullet  \arrow[ldd] \arrow[rdd] \arrow[d] & \\
& \bullet & \\
\bullet \arrow[ru]\arrow[rr] \arrow[ruu] & & \bullet \arrow[lu] \arrow[luu] \arrow[ll]
\end{tikzcd} \qquad W_4''=\begin{tikzcd}[row sep=0.5cm, column sep=0.5cm]
& \bullet  \arrow[ldd] \arrow[rdd]  & \\
& \bullet \arrow[u] \arrow[dl] \arrow[dr] & \\
\bullet \arrow[rr] \arrow[ruu] & & \bullet  \arrow[luu] \arrow[ll] .
\end{tikzcd}
\end{equation*}
\end{myeg}

\begin{proposition}\label{proposition: wheel1}
Let $L_{W_n}$ (resp.\ $L_{W_n'}$) be the Laplacian matrix of $W_n$ (resp.\ $W_n'$). Then $L_{W_n}$ and $L_{W_n'}$ are row equivalent. In particular, one has $\Pic (W_n) \cong \Pic (W_n')$.
\end{proposition}

\begin{proof}
We use the labeling convention that $v_1$ is the axle and the vertices on the rim are $v_2, \dots, v_{n}$. The Laplacian matrices are as follows:
\begin{equation*}
L_{W_n} =\scalemath{0.85}{
\left[
\begin{array}{c|cccccccc}
n-1 & -1 & -1 & \cdots & \cdots & \cdots & \cdots & -1 \\
\hline
-1 & 3 & -1 & 0 & \cdots & \cdots & 0 & -1 \\
-1 & - 1& 3 & -1 & \cdots & \cdots & 0 & 0 \\
-1 & 0 & - 1& 3 & -1 & \cdots & 0 & 0 \\
-1 & 0 & 0 & - 1& 3 & -1 & \cdots & 0 \\
-1 & \vdots & \vdots &\vdots & \vdots&\vdots &\vdots & \vdots\\
-1 & 0 & 0 & 0 & \cdots & -1& 3 & -1 \\
-1 & -1 & 0 & 0 & \cdots& 0 & -1 & 3 
\end{array}
\right]}
\quad 
L_{W_n'} = \scalemath{0.85}{
\left[
\begin{array}{c|ccccccc}
0 & 0 & 0 & \cdots & \cdots & \cdots & \cdots & 0 \\
\hline
-1 & 3 & -1 & 0 & \cdots & \cdots & 0 & -1 \\
-1 & - 1& 3 & -1 & \cdots & \cdots & 0 & 0 \\
-1 & 0 & - 1& 3 & -1 & \cdots & 0 & 0 \\
-1 & 0 & 0 & - 1& 3 & -1 & \cdots & 0 \\
-1 & \vdots & \vdots &\vdots & \vdots&\vdots &\vdots & \vdots\\
-1 & 0 & 0 & 0 & \cdots & -1& 3 & -1 \\
-1 & -1 & 0 & 0 & \cdots& 0 & -1 & 3 
\end{array}
\right]}.
\end{equation*}
Since $\mathbf{1_n} \cdot L_{w_n} = 0$ (\Cref{laplaceFactUndirected}), 
$L_{W_n'}$ is obtained from $L_{W_n}$ by adding all other rows from the first row.
This completes the proof. 
\end{proof}

\begin{proposition}\label{proposition: wheel2}
The Smith normal form of the Laplacian of $W_n''$ is of the following form.
\begin{equation*}
\left[
\begin{array}{c|ccc}	
I_{n-3} & 0 & 0 & 0 \\
\hline
0 & a & 0 & 0 \\
0 & 0 & b & 0 \\
0 & 0 & 0 & 0 
\end{array}
\right], \quad ~\textrm{where}~ (a,b) = \begin{cases}
(n-1,n-1) & \text{if}~ n ~\text{is even}; \\
(\frac{n-1}2, 2(n-1)) & \text{if}~ n ~\text{is odd}.
\end{cases}
\end{equation*}
\end{proposition}
\begin{proof}
We use the labeling convention that $v_1$ is the axle and the vertices on the rim are $v_2, \dots, v_{n}$. 
Then we have the following Laplacian matrix $L_{W_n''}$ of $W_n''$. 
\begin{equation*}
L_{W_n''} =
\left[
\begin{array}{c|cccccccc}
n-1 & -1 & -1 & \cdots & \cdots & \cdots & \cdots & -1 \\
\hline
0 & 2 & -1 & 0 & \cdots & \cdots & 0 & -1 \\
0 & - 1& 2 & -1 & \cdots & \cdots & 0 & 0 \\
0 & 0 & - 1& 2 & -1 & \cdots & 0 & 0 \\
0 & 0 & 0 & - 1& 2 & -1 & \cdots & 0 \\
0 & \vdots & \vdots &\vdots & \vdots&\vdots &\vdots & \vdots\\
0 & 0 & 0 & 0 & \cdots & -1& 2 & -1 \\
0 & -1 & 0 & 0 & \cdots& 0 & -1 & 2 
\end{array}
\right].
\end{equation*}
Let $r_1,\dots,r_n$ and $c_1,\dots,c_n$ denote the rows and columns of $L_{W_n''}$.
First, $I_{n-3}(L_{W_n''})=\angles{1}$. 
This follows by taking the $(n-3) \times (n-3)$ minor after deleting $r_1,r_2,r_n$ and $c_1,c_{n-1},c_n$.
By \Cref{kernelLaplician}, we have $\det(L_{W_n''})=0$. 
Hence, by Theorem \ref{theorem: gcd theorem}, it suffices to show that 
\begin{equation*}
I_{n-1}(L_{W_n''}) = \langle (n-1)^2 \rangle~\text{and}~ I_{n-2}(L_{W_n''}) = \begin{cases}
\langle n-1 \rangle & \text{if}~ n ~\text{is even} \\
\langle (n-1)/2\rangle & \text{if}~ n ~\text{is odd}.
\end{cases}
\end{equation*}
Now, we reduce the Laplacian matrix $L_{W_n''}$.  
By replacing $c_1$ with $c_1+\cdots+c_n$, multiplying $-1$ to $r_1$, and then replacing $r_n$ with $r_2+\cdots +r_n$, we obtain the following matrix. 
\begin{equation}\label{wppIntMat}
\left[	\begin{array}{c|ccccccc}
0 & 1 & 1 & \cdots & \cdots & \cdots & \cdots & 1 \\
\hline
0  & 2 & -1 & 0 & \cdots & \cdots & 0 & -1 \\
0  & - 1& 2 & -1 & \cdots & \cdots & 0 & 0 \\
0  & 0 & - 1& 2 & -1 & \cdots & 0 & 0 \\
0 & 0 & 0 & - 1& 2 & -1 & \cdots & 0 \\
0  & \vdots & \vdots &\vdots & \vdots& \vdots& \vdots& \vdots\\
0  & 0 & 0 & 0 & \cdots & -1& 2 & -1 \\
0  & 0 & 0 & 0 & \cdots& 0 & 0 & 0
\end{array} \right].
\end{equation}
By abuse of notation, we use $c_1,\dots, c_n$ to denote the columns of the matrix in \cref{wppIntMat}.
After replacing $c_n$ with $c_2+\cdots+c_n$, we obtain the following matrix.
\begin{equation}
\left[\begin{array}{c|ccccccc}
0  & 1 & 1 & \cdots & \cdots & \cdots & 1 & n-1 \\
\hline
0  & 2 & -1 & 0 & \cdots & \cdots & 0 & 0 \\
0  & - 1& 2 & -1 & \cdots & \cdots & 0 & 0 \\
0  & 0 & - 1& 2 & -1 & \cdots & 0 & 0 \\
0  & 0 & 0 & - 1& 2 & -1 & \cdots & 0 \\
0  & \vdots & \vdots &\vdots & \vdots& \vdots& \vdots& \vdots\\
0  & 0 & 0 & 0 & \cdots & -1& 2 & 0 \\
0 & 0 & 0 & 0 & \cdots& 0 & 0 & 0
\end{array} \right].
\end{equation}
Since the first column and the last row are zero, 
to compute the Smith normal form of $L_{W_n''}$, it suffices to find the Smith normal form of the $(n-1)$ by $(n-1)$ matrix $N$.
\begin{equation*}
N := \left[
\begin{array}{c|c}
\mathbf{1}_{n-2} & n-1  \\
\hline \\[-0.4cm]
M_{n-2} & \mathbf{0}_{n-1}^T
\end{array}
\right],
\end{equation*}
where $M_{n-2}$ is the matrix \cref{eq: m_n} in Lemma \ref{lemma: obj3}. 
Note that $\det(M_{n-2}) = n-1$ and this shows that
\begin{equation*}
I_{n-1}(L_{W_n''}) = I_{n-1}(N) =\angles{\det(N)}=\angles{(n-1)^2}.
\end{equation*} 
To compute the $(n-2)$ minors, first note that if the last column is a part of a minor, then it is divisible by $(n-1)$. 
Since $\det(M_{n-2}) = (n-1) \in I_{n-2}(N)$, we only need to consider the $(n-2)$ minors of the first $(n-2)$ columns of $N$. In other words, $I_{n-2}(L_{W_n''})$ is generated by $(n-1)$ 
and the $(n-2)$-determinants of the matrix which is obtained from $M_{n-2}$ by replacing the $i^{\textrm{th}}$ row by $\mathbf{1}_{n-2}$.

Note that $M_{n-2}$ is symmetric. 
By Cramer's rule, these $(n-2)$ minors are the entries of the solution matrix $\mathbf{x}$ of the following matrix equation (up to sign):
\begin{equation}\label{eq: minors}
M_{n-2} \mathbf{x} 
= \det(M_{n-2}) \cdot \mathbf{1}_{n-2}^T.
\end{equation}
By \Cref{lemma: M_nCremer}, 
the greatest common divisor of the solutions to \cref{eq: minors}
is $(n-1)$ if $n$ is even and $(n-1)/2$ if $n$ is odd. 
This completes the proof.
\end{proof}

\begin{lemma}\label{lemma: M_nCremer}
Consider the matrix equation 
\begin{equation*}
M_n \mathbf{x} 
= (n+1) \mathbf{1_n}^T,
\end{equation*}
where $M_n$ is as in Lemma \ref{lemma: obj3}. 
Let $x_k$ denote the $k^{\textrm{th}}$ entry of $\mathbf{x}$. Then $x_k = a k^2 + b k$, where
\begin{equation*}
a = -\frac{(n+1)}{2}, \quad b = \frac{(n+1)^2}{2}.
\end{equation*}
Furthermore, $\gcd(x_1,\dots, x_n)$ is $(n+1)$ if $n$ is even and $(n+1)/2$ if $n$ is odd. 
\end{lemma}

\begin{proof}
Note that $\det (M_n) = n+1$ by Lemma \ref{lemma: obj3}. Consider a sequence $f(k)=f_k  = ak^2 + bk$. 
For any $k$, we have
\begin{equation*}
-f_{k-1} + 2f_k - f_{k+1} = -2a.
\end{equation*}
First, we set $a = -(n+1)/2$. 
In addition, the equations $2x_1 - x_2 = n+1$ and $-x_{n-1} + 2x_n = n+1$ are equivalent to the conditions $f_0 = f_{n+1} = 0$. 
These conditions imply $b = (n+1)^2/2$ since $0, -b/a$ are the roots of $f(k)$. 

We compute the gcd of $f_1,\dots, f_n$. 
First, we claim that $\gcd(f_1,\dots, f_n) = \gcd(f_1, n+1)$. 
Since $\gcd(f_1,\dots, f_n)$ divides $2f_1-f_2 = n+1$, 
it suffices to show that $\gcd(f_1, n+1)$ divides $f_2,\dots, f_n$. 
This follows by induction since for $k = 2, \dots, n$, we have
\[
f_k = 2 f_{k-1} - f_{k-2} - (n+1).
\]
Here we use the fact that $f_0 = 0$ for the base case of the induction. 
Moreover, since
\[
f_1 =  -(n+1)/2 + (n+1)^2/2 = \frac{n(n+1)}2
\]
is an integer, the expressions $f_k = -2 f_{k-1} + f_{k-2}$ also prove that $f_1, \dots, f_n$ are integers. 

Finally, we show that $\gcd(f_1, n+1)$ is $n+1$ when $n$ is even and $(n+1)/2$ when $n$ is odd. 
Recall $f_1 = \frac{n(n+1)}2$.
If $n$ is even, then $n+1$ divides $f_1$. 
Thus, $\gcd(f_1, n+1) = n+1$. 
If $n$ is odd, then 
\begin{equation*}
\gcd(f_1, n+1) = \gcd(\frac{n+1}2 \cdot n, n+1) = \gcd(\frac{n+1}2, n+1) = \frac{n+1}2.
\end{equation*} 
The second last equality follows from the fact that $\frac{n+1}2$ is an integer and for integers $a,b,c$, $\gcd(ab,c) = \gcd(a,c)$ if $b,c$ are relatively prime. 

Thus, $\mathbf{x} = \begin{bmatrix} f_1 & \cdots & f_n \end{bmatrix}^T$ is a solution having the asserted gcd. 
This completes the proof. 
\end{proof}

\section{On certain directed multipartite graphs} \label{section: conjecture}

In this section, we study Picard groups of a class of directed multipartite graphs. 
We call members of this class as single-flow directed multipartite graphs. 
The structure of the graphs that we investigate are designed to resemble artificial neural networks. 
Our results 
provides patterns in for \emph{Perceptron} style model with two layers (\Cref{bipartitePic}) and a \emph{Hidden Layer} model with three layers (\Cref{multipartite3Pic}).
Thus, \Cref{multipartiteQuestion} can be though of as an attempt to understand the ``deep'' neural networks.  

\begin{mydef}
By a \emph{single-flow directed multipartite graph} $(G,\mathcal{P})$, we mean a directed graph $G$ together with a partition $\mathcal{P}$ of the vertex set $V(G)=V_1(G) \sqcup \cdots \sqcup V_t(G)$ satisfying the following conditions.
\begin{enumerate}
\item There is no empty partition. That is, $|V_i(G)| > 0$ for all $i \in \{1, \dots, t \}$. 
\item 
For any $i \in \{1,\dots,t\}$ and $u,v \in V_i(G)$, there is no arrow between $u$ and $v$,
\item 
For any $i  \in \{1,\dots,t-1\}$, $u \in V_i(G)$, and $v \in V_{i+1}(G)$, there exists an arrow $e_{\overrightarrow{u v}}$.
\end{enumerate}	
For a single-flow directed multipartite graph $(G,\mathcal{P})$, we call $t$ the \emph{number of layers} of $(G,\mathcal{P})$. 
\end{mydef}

The Picard group of a undirected multipartite graph is determined in \cite{jacobson2003critical}. 

\begin{myeg}
When $a=2$ and $b=3$, we have the following:
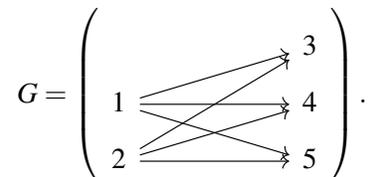
\[
G=\left(\begin{tikzcd}[row sep=0.2cm, column sep=1cm]
& & 3 \\
1 \arrow[rru] \arrow[rr] \arrow[rrd]& &4 \\
2 \arrow[rru] \arrow[rr] \arrow[rruu]& &5
\end{tikzcd} \right).
\]
Then we have
\[
L_G = \left[ \begin{array}{cc|ccc}
3& 0& -1&-1 &-1 \\
0 & 3 & -1 & -1 &-1\\
\hline
0& 0& 0& 0 &0\\
0& 0& 0 &0 & 0\\
0& 0&0&0 & 0
\end{array}
\right] \implies \textrm{SNF}(L_G) = \begin{bmatrix}
1& 0& 0&0 &0\\
0 & 3 & 0 & 0&0\\
0& 0& 0& 0 &0\\
0& 0& 0 &0 & 0\\
0& 0&0&0 & 0
\end{bmatrix}
\]
Hence $\Pic(G) \cong \Z^3 \times \Z_3$.
\end{myeg}

\begin{pro}\label{bipartitePic}
Let $(G,\mathcal{P})$ be a single-flow directed multipartite graph with two layers. Let $|V_1(G)|=a$ and $|V_2(G)|=b$. Then we have
\[
\Pic(G) \cong \mathbb{Z}^b \times \mathbb{Z}_{b}^{a-1}.
\]
\end{pro}
\begin{proof}
The Laplacian matrix of $G$ is the following form.	
\[
L_G=\left[ \begin{array}{c|c}
bI_{a} & - \mathbf{1}_{a\times b} \\
\hline 
\mathbf{0} & \mathbf{0}
\end{array} \right].
\] 
By subtracting the $(a+1)$th column from the columns right to it, 
$L_G$ is equivalent to the following matrix.
\[
\left[ \begin{array}{c|c|c}
bI_{a} & -\mathbf{1}_a^T & \mathbf{0} \\
\hline 
\mathbf{0} & \mathbf{0} & \mathbf{0}
\end{array} \right].
\]
Note that $I_k(L_G) = I_k(L_{sub})$, where $L_{sub} = \begin{bmatrix} b I_a & \mathbf{1}_a^T \end{bmatrix}$.
Thus, it suffices to study $L_{sub}$. Let $c_1, \dots, c_{a+1}$ denote the columns of $L_{sub}$. Replace $c_1$ with $-(c_2+\cdots+c_a)+bc_{a+1}$ and then switch the first and the last columns. One has the following equivalences of $L_{sub}$. 
\begin{align}\label{multipartiteEq}
\begin{bmatrix}
b & 0 & \cdots  &\cdots  &0 &1 \\
0 & b & 0 & \cdots & 0 & 1 \\
0 & 0 & \ddots & 0 & 0 & 1\\
\vdots &  \vdots & \vdots & \ddots & \vdots &\vdots \\
0 &  \cdots& 0 & 0 & b & 1  
\end{bmatrix} 
\sim 
\begin{bmatrix}
0 & 0 & \cdots  &\cdots  &0 &1 \\
0 & b & 0 & \cdots & 0 & 1 \\
0 & 0 & \ddots & 0 & 0 & 1\\
\vdots &  \vdots & \vdots & \ddots & \vdots &\vdots \\
0 &  \cdots& 0 & 0 & b & 1  
\end{bmatrix} 
&
\sim 
\begin{bmatrix}
1 & 0 & \cdots  &\cdots  &0 &0 \\
1 & b & 0 & \cdots & \cdots & 0 \\
1 & 0 & \ddots & 0 & \cdots & 0\\
\vdots &  \vdots & \vdots & \ddots & \vdots &\vdots \\
1 &  \cdots& 0 & 0 & b & 0  
\end{bmatrix} 
\end{align}
By subtracting the first row from all other rows, we have the following matrix.
\[
\left[ \begin{array}{c|c|c}
1 & \mathbf{0} & \mathbf{0}\\
\hline 
\mathbf{0} & bI_{a-1} &\mathbf{0}
\end{array} \right].	
\]
Now, the result follows from the Smith normal form of $L_G$.
\begin{equation*}
\snf(L_G)
= 
\left[ \begin{array}{c|c|c}
1 & \mathbf{0} & \mathbf{0} \\
\hline
\mathbf{0} & bI_{a-1} & \mathbf{0} \\
\hline 
\mathbf{0} & \mathbf{0} & \mathbf{0}
\end{array} \right]. \qedhere
\end{equation*}
\end{proof}

\begin{remark}
Let $(G,\mathcal{P})$ be a single-flow directed multipartite graph with $t$ layers and $|V_i(G)| = a_i$. Then the Laplacian matrix $L_G$ is of the following block matrix form
\begin{equation*}
\left[ \begin{array}{c|c|c|c|c|c}
a_2 I_{a_1} & -\mathbf{1}_{a_1 \times a_2} & \mathbf{0} & \mathbf{0} & \cdots  & \mathbf{0}\\
\hline 
\mathbf{0} & a_3 I_{a_2} & -\mathbf{1}_{a_2 \times a_3} & \cdots & \mathbf{0} & \mathbf{0}\\
\hline  
\vdots & \vdots & \ddots & \ddots & \vdots & \vdots  \\
\hline
\vdots & \vdots & \vdots & \ddots & \ddots & \vdots  \\
\hline
\mathbf{0} & \mathbf{0} & \cdots & \mathbf{0} & a_{t} I_{a_{t-1}} & -\mathbf{1}_{a_{t-1} \times a_t}  \\
\hline
\mathbf{0} & \mathbf{0} & \cdots & \mathbf{0} & \mathbf{0} & \mathbf{0}
\end{array} \right],
\end{equation*}
where the $(i,j)$th block is of size $a_i \times a_j$.  

From this one can deduce that the rank of the Picard group of $G$ is $a_t$.\footnote{One can also see in this case that each vertex in $V_t(G)$ is a terminal strong component. Hence, this directly follows from Theorem \ref{theorem: wagner}.}
This follows since the number of non-zero rows is $r := a_1 + \cdots + a_{t-1}$ and the determinant of the principal $r \times r$ minors is $a_2^{a_1} \cdots a_t^{a_{t-1}} \neq 0$. 
Here the principal minor of size $k$ of a matrix $M$ is the minor of the first $k$ rows and columns. 
We note that this minor is not equal to $|\Jac(G)|$ in general, see \Cref{sizeOfJac3}.
\end{remark}

\begin{lemma}\label{sizeOfJac3}
Let $(G,\mathcal{P})$ be a single-flow directed multipartite graph with three layers, and let $|V_1(G)|=a$, $|V_2(G)|=b$, and $|V_3(G)|=c$. 
Then $|\Jac(G)| = {b^a \cdot c^{b-1}}$. 
\end{lemma}
\begin{proof}
We have the following Laplacian matrix of $G$. 
\begin{equation*}
L_G = \left[ \begin{array}{c|c|c}
b I_{a} & -\mathbf{1}_{a \times b} & \mathbf{0} \\
\hline 
\mathbf{0} & c I_{b} & -\mathbf{1}_{b \times c} \\
\hline
\mathbf{0} & \mathbf{0} & \mathbf{0}
\end{array} \right].
\end{equation*}
Since the rank of $L_G$ is $a+b$, 
it suffices to show that $I_{a+b}(L_G) = \langle b^a \cdot c^{b-1} \rangle$.
First observe that 
\begin{equation*}
I_{a+b}(L_G) = I_{a+b}\left( 
\left[ \begin{array}{c|c|c}
b I_{a} & -\mathbf{1}_{a \times b} & \mathbf{0} \\
\hline  \\[-0.4cm]
\mathbf{0} & c I_{b} & \mathbf{1}_{b}^T 
\end{array} \right]	
\right).
\end{equation*}
We apply the modifications in \cref{multipartiteEq} to the second row block of $L_G$. 
Note that in this case we also need to keep track of the $(1,2)$ and $(1,3)$ sub-blocks. 
\begin{align*}
\left[ \begin{array}{c|c|c}
b I_{a} & -\mathbf{1}_{a \times b} & \mathbf{0} \\
\hline  \\[-0.4cm]
\mathbf{0} & c I_{b} & \mathbf{1}_{b}^T 
\end{array} \right]	
\sim
\left[ \begin{array}{c|c|c}
b I_{a} & N & b\cdot\mathbf{1}_a^T \\
\hline  
\mathbf{0} & N' & \mathbf{0}
\end{array} \right],
\end{align*}
where $N$ and $N'$ are the following matrices.
\begin{equation*}
N = \left[ \begin{array}{c|c}
\mathbf{0}_a^T & -\mathbf{1}_{a \times (b-1)}
\end{array} \right]	
~\text{and}~
N' =  \left[ \begin{array}{c|c}
1 & \mathbf{0}  \\
\hline
\mathbf{0} & cI_{b-1} 
\end{array} \right]	
\end{equation*}
Note that the last column is the sum of the first $a$ columns. 
Therefore, $I_{a+b}(L) = \langle \det(bI_a) \cdot \det(N') \rangle = \langle b^a \cdot c^{b-1} \rangle$ and $|\Jac(G)| = {b^a \cdot c^{b-1}}$.
\end{proof}

Indeed, one can further reduce to a simpler case by reducing the block matrix
\begin{equation*}
L=\left[ \begin{array}{c|c}
b I_{a} & N  \\
\hline 
\mathbf{0} & N' 
\end{array} \right],
\end{equation*}
and it will provide a way to determine the Smith normal form of $L_G$. 

Since the $(1,1)$-entry of $N'$ is $1$, 
$I_{k+1}(L) = I_{k}(L')$, where 
$L'$ is the matrix obtained by deleting the row and column of containing the $(1,1)$ entry of $N'$.
That is,  
\begin{equation*}
L'=\left[ \begin{array}{c|c}
b I_{a} & -\mathbf{1}_{a \times (b-1)}  \\
\hline 
\mathbf{0} & cI_{b-1} 
\end{array} \right].
\end{equation*}
If $b = 1$, then by \Cref{sizeOfJac3}, $|\Jac(G)| = 0$. 
This implies that the non-zero diagonal entries of $\snf(L_G)$ consists of $1$'s. 
In the rest, we assume that $b > 1$, so the expression $I_{b-2}$ is valid.
Let $r_1, \dots, r_{a+b-1}$ and $c_1, \dots, c_{a+b-1}$ denote the rows and columns of $L'$. 
By subtracting $r_a$ from each of $r_1, \dots, r_{a-1}$ and by subtracting $c_{a+1}$ from each of 
$c_{a+2}, \dots, c_{a+b-1}$, 
we have
\begin{equation*}
L'\sim \left[ \begin{array}{ccccc|ccccc}
b & 0 & \cdots & 0 & -b &0&\cdots&\cdots&\cdots&0  \\ 
0 & b & \cdots & 0 & -b &\vdots&\cdots&\cdots&\cdots&\vdots  \\ 
\vdots & \vdots & \ddots & \vdots & \vdots &
\vdots & \vdots & \ddots & \vdots & \vdots\\ 
0 & \cdots & 0 & b & -b &0&0&\cdots&\cdots&0  \\ 
0 & \cdots & \cdots & 0 & b &-1& 0 &\cdots&\cdots&0  \\ 
\hline
0&\cdots&\cdots&\cdots&0 & c & -c & -c & \cdots & -c \\ 
\vdots&\cdots&\cdots&\cdots&\vdots & 0 & c & \cdots & 0 & 0 \\ 
\vdots & \vdots & \ddots & \vdots & \vdots & \vdots & \vdots & \ddots & \vdots & \vdots \\ 
\vdots&\cdots&\cdots&\cdots&\vdots & 0 & \cdots & 0 & c & 0 \\ 
0&\cdots&\cdots&\cdots&0 & 0 & \cdots & \cdots & 0 & c \\ 
\end{array} \right].
\end{equation*}	
Further column and row operations will reduce to the case of having $bI_a$ for the $(1,1)$-block and $cI_{b-1}$ for the $(2,2)$-block, respectively.
Last reduction is done on the sub-block 
$\begin{bmatrix} b & -1 \\ 0 & c \end{bmatrix}$ 
by replacing it with its Smith normal form 
$\begin{bmatrix} 1 & 0 \\ 0 & bc \end{bmatrix}$.
Therefore, $L'$ is equivalent to the following matrix
\begin{equation*}
L'' = \left[ \begin{array}{c|c|c|c}
1 & 0 & \mathbf{0} & \mathbf{0}\\
\hline
0 & bc & \mathbf{0} & \mathbf{0}\\
\hline
\mathbf{0} & \mathbf{0} & bI_{a-1} & \mathbf{0} \\
\hline
\mathbf{0} & \mathbf{0} & \mathbf{0} & cI_{b-2}
\end{array} \right].
\end{equation*}
Hence $L_G$ is equivalent to the following matrix
\begin{equation*}
L_G \sim L_G' = \left[ \begin{array}{c|c|c|c|c}
I_2 & \mathbf{0} & \mathbf{0} & \mathbf{0}& \mathbf{0}\\
\hline
0 & bc & \mathbf{0} & \mathbf{0} & \mathbf{0}\\
\hline
\mathbf{0} & \mathbf{0} & bI_{a-1} & \mathbf{0} & \mathbf{0}\\
\hline
\mathbf{0} & \mathbf{0} & \mathbf{0} & cI_{b-2} & \mathbf{0} \\
\hline
\mathbf{0} & \mathbf{0} & \mathbf{0} & \mathbf{0} & \mathbf{0}
\end{array} \right].
\end{equation*}	
Though $L_G'$ is a diagonal matrix, it is not in its Smith normal form in general. 
However, since $L_G'$ is a diagonal matrix, we can determine the cokernel of $L_G$ from $L_G'$. 
That is, 
\begin{equation}\label{eleDivMul3}
\operatorname{coker} L_G \cong \Pic (G) \cong \Z_{bc} \times \Z_b^{a-1} \times \Z_c^{b-2} \times \mathbb{Z}^c.
\end{equation} 
This determines the Picard group of $G$. 
We note that by the well-known correspondence between the elementary divisors and invariant factors of a finitely generated abelian group, 
one can deduce the Smith normal form of $L_G$ from \cref{eleDivMul3}. 
Another method is direct though it is essentially the same as the correspondence.
Note that the Smith normal form of the matrix $\begin{bmatrix}
b & 0 \\ 0 & c
\end{bmatrix}
$
is $\begin{bmatrix}
\gcd(b,c) & 0 \\ 0 & \operatorname{lcm}(b,c)
\end{bmatrix}$,
where $\gcd$ and $\operatorname{lcm}$ denotes the greatest common divisor and the least common multiple. 
Therefore, the non-zero diagonal entries of the Smith normal form of $L_G$ consists of $1, \gcd(b,c), \allowbreak \operatorname{lcm}(b,c), b, c$, and $bc$.
It is clear that $1 \mid \gcd(b,c) \mid b,c \mid \operatorname{lcm}(b,c) \mid bc$. 
Recall if $b = 1$, then the non-zero diagonal entries of $\snf(L_G)$ consists of $1$'s. 
In \cref{minorLMul3}, we write $g = \gcd(b,c)$ and $l = \operatorname{lcm}(b,c)$, and let the exponential notation $a^{[b]}$ denote $\underbrace{a,\dots, a}_{b-\text{times}}$. 
With this notation, the non-zero diagonal entries of the Smith normal form of $L_G$ follow the following pattern. 
\begin{equation}\label{minorLMul3}
\begin{array}{ll}
1^{[a+b]} & \text{if}~b = 1, \\
1^{[2]},g^{[(b-2)]},b^{[a-b+1]},l^{[(b-2)]},bc & \text{if}~a-1 \ge b-2 \ge 0, ~\text{and}\\
1^{[2]},g^{[(a-1)]},c^{[b-a-1]},l^{[(a-1)]},bc & \text{if}~a-1 \le b-2.
\end{array}
\end{equation}
Thus, we have proved the following theorem.

\begin{theorem}\label{multipartite3Pic}
Let $(G,\mathcal{P})$ be a single-flow directed multipartite graph with three layers, and let $|V_1(G)|=a$, $|V_2(G)|=b$, and $|V_3(G)|=c$. 
Then the Picard group of $G$ is isomorphic to 
$$\Z^c \times \Jac(G),$$ where 
\begin{equation*}
\Jac(G) \cong
\begin{cases}
0 & \text{if}~b = 1, \\
\Z_g^{b-2}\times \Z_b^{a-b+1}\times \Z_l^{b-2}\times \Z_{bc} & \text{if}~a \ge b-1 ~\text{and}~b\ge 2, \\
\Z_g^{a-1}\times \Z_c^{b-a-1}\times \Z_l^{a-1} \times \Z_{bc} & \text{if}~a \le b-1.	
\end{cases}
\end{equation*}
Here, $g$ and $l$ denote the greatest common divisor and the least common multiple of $b$ and $c$, respectively. 
\end{theorem}

The structure of the graphs that we investigate are designed to resemble artificial neural networks. 

\begin{question}\label{multipartiteQuestion}
Let $(G,\mathcal{P})$ be a single-flow directed multipartite graph with $t$ layers.
What is the Smith normal form of the Laplacian matrix $L_G$ of $G$? 
That is, what is their Picard groups? 
\end{question}

\bigskip
\bibliographystyle{plain} 
\bibliography{Jacobian}
\end{document}